\documentclass{article}

\usepackage{amsmath,amsthm,amssymb,tikz,bbm,float}
\usetikzlibrary{arrows}
\usepackage[margin=1.2in]{geometry}

\newtheorem{theorem}{Theorem}[subsection]

\newtheorem{lemma}[theorem]{Lemma}
\newtheorem{proposition}[theorem]{Proposition}

\theoremstyle{definition}
\newtheorem{definition}[theorem]{Definition}

\newtheorem{example}[theorem]{Example}
\newtheorem{note}[theorem]{Note}
\newtheorem{question}[theorem]{Question}
\newtheorem*{acknowledgements*}{Acknowledgements}

\title{Nondegenerate extensions of near-group \\ braided fusion categories}
\author{Andrew Schopieray}
\date{\small
    	Department of Mathematical and Statistical Sciences\\
	CAB 632 \\
	University of Alberta \\
	Edmonton, Alberta \\
	Canada T6G 2G1 \\
	\texttt{schopier@ualberta.ca}
}

\begin{document}

\maketitle

\begin{abstract}
This is a study of weakly integral braided fusion categories with elementary fusion rules to determine which possess nondegenerately braided extensions of theoretically minimal dimension, or equivalently in this case, which satisfy the minimal modular extension conjecture.  We classify near-group braided fusion categories satisfying the minimal modular extension conjecture; the remaining Tambara-Yamagami braided fusion categories provide arbitrarily large families of braided fusion categories with identical fusion rules violating the minimal modular extension conjecture.  These examples generalize to braided fusion categories with the fusion rules of the representation categories of extraspecial $p$-groups for any prime $p$, which possess a minimal modular extension only if they arise as the adjoint subcategory of a twisted double of an extraspecial $p$-group.
\end{abstract}

\section{Introduction}

A braided fusion category embeds into its center, or \emph{double} \cite[Definition 7.13.1]{tcat}, allowing specialized results about nondegenerately braided fusion categories to be used in generality.  Unfortunately, the double is a large construction.  The categorical and Frobenius-Perron dimensions are squared, and the rank of the double has yet to be bound in terms of the rank of the original.  So given a braided fusion category $\mathcal{C}$, it is desirable to seek nondegenerately braided fusion categories $\mathcal{D}$ containing $\mathcal{C}$ as a braided fusion subcategory which are smaller than the double in any sense.  For dimension, the theoretical minimum that can be accomplished is $\mathrm{FPdim}(\mathcal{D})=\mathrm{FPdim}(\mathcal{C})\mathrm{FPdim}(C_\mathcal{C}(\mathcal{C}))$ where $C_\mathcal{C}(\mathcal{C})$ is the symmetric center of $\mathcal{C}$ (see Section \ref{braid}).  Whether a nondegenerately braided extension of theoretically minimal dimension exists is a natural question which is approximately 20 years old now \cite[Conjecture 5.2]{mug1}.  This question is often called the \emph{minimal modular extension conjecture} as it was originally stated with the assumption of a spherical structure; the minimal modular extension conjecture is false with or without this assumption.  The first counterexamples were explained by V.\ Drinfeld in private communications which have not appeared in print at this time.  Obstructions to the existence of certain minimal modular extensions may be found in cohomological data of finite groups and this line of reasoning was used by C.\ Galindo and C.\ F.\ Venegas-Ram\'irez in \cite[Section 4.3]{galindo2017categorical} to provide novel counterexamples to the minimal modular extension conjecture.  The smallest of these counterexamples \cite[Proposition 4.11]{galindo2017categorical} is a fusion category of rank 5 and dimension 8 which is easily seen to be the representation category of the dihedral group of order 8 with a nonsymmetric braiding.  In Proposition \ref{bigprop} we prove this is the first in an infinite family of counterexamples coming from the Tambara-Yamagami braided fusion categories.  One way in which these arguments are novel is that they only rely on the fusion rules of the category, giving arbitrarily many examples with the same fusion rules as their rank is increased.  The Tambara-Yamagami story is a large portion of the proof of Theorem \ref{theorem}: a classification of near-group braided fusion categories which possess a minimal modular extension.  With 5 exceptions (Figure \ref{fig:onetwotwo}) these are realized as the unique fusion subcategory of a product of Ising braided fusion categories $\otimes$-generated by the simple object of maximal dimension.

\begin{figure}[H]
\centering
\begin{equation*}
\begin{array}{|c|ccc|}
\hline \mathcal{C} & \mathrm{rank}(\mathcal{C}) & \mathrm{FPdim}(\mathcal{C}) & \mathrm{FPdim}(C_\mathcal{C}(\mathcal{C}))\\\hline\hline
\mathcal{C}(\mathfrak{sl}_2,4)_\mathrm{ad} & 3 &6 & 2  \\
\mathcal{C}(\mathfrak{sl}_2,4)^\mathrm{rev}_\mathrm{ad} & 3 & 6 & 2 \\
\mathcal{Z}(\mathrm{Vec}^\gamma_{Q_8})_\mathrm{ad} & 5 & 8 & 4  \\
\mathcal{Z}(\mathrm{Vec}^\gamma_{Q_8})^\mathrm{rev}_\mathrm{ad} & 5 &8 & 4 \\
\mathcal{C}(\mathfrak{sl}_3,3)_\mathrm{ad} & 4 & 12 & 3 \\
\hline
\end{array}
\end{equation*}
    \caption{Exceptional near-group braided fusion categories which possess minimal modular extensions; $\gamma$ is any generator of $H^3(Q_8,\mathbb{C}^\times)$}%
    \label{fig:onetwotwo}%
\end{figure}

\par The fusion rules of integral Tambara-Yamagami braided fusion categories coincide with the character rings of extraspecial $2$-groups (see Section \ref{extraspecial}).  In general, for each prime $p$ and $n\in\mathbb{Z}_{\geq1}$, there exist nonsymmetrically braided fusion categories whose fusion rules coincide with the character rings of $p_{\pm}^{1+2n}$, where $p_{\pm}^{1+2n}$ is either of the extraspecial $p$-groups of order $p^{2n+1}$.  It is an open problem to classify fusion categories with these fusion rules for $p\neq2$, as well as to classify the compatible braidings.  One computational hurdle is exponentially increasing multiplicities in the fusion rules as $n$ increases.  This problem is of independent interest, but these nonsymmetrically braided fusion categories should provide yet another infinite class of counterexamples to the minimal modular extension conjecture.  We prove in Proposition \ref{3andon} that a nonsymmetrically braided fusion category whose fusion rules coincide with those of the character ring of $p_{\pm}^{1+2n}$ for some $p\neq2$ and $n\in\mathbb{Z}_{\geq1}$ possesses a nondegenerately braided extension of theoretically minimal dimension only if it arises as the adjoint subcategory of a twisted double of an extraspecial $p$-group.

\par Our exposition is divided into five additional sections.  Section \ref{pre} describes the notation and vocabulary used in the remainder of the sections and provides the reader with further resources.  Section \ref{interesting} describes nondegenerately braided extensions for $(\mathbb{Z}/2\mathbb{Z})^{\oplus2}$ Tambara-Yamagami braided fusion categories in explicit detail, as this is the case where two important exceptions occur.  The goal of Section \ref{jen} is to prove that $(\mathbb{Z}/2\mathbb{Z})^{\oplus n}$ Tambara-Yamagami categories are more uniform for $n>1$, with a large proportion lacking minimal modular extensions.  Section \ref{near} utilizes the existing classification of braided near-group fusion categories due to J.\ Thornton \cite[Theorem III.4.6]{MR3078486} to give a complete classification of braided near-group fusion categories satisfying the minimal modular extension conjecture.  Lastly, Section \ref{futuro} contains preliminary results which should allow the content of Section \ref{jen} to be generalized to arbitrary extraspecial $p$-groups.


\section{Preliminaries}\label{pre}

\subsection{Braided fusion categories}\label{braid}

The basic object in what follows is a fusion category \cite[Section 2]{ENO}.  Our exposition roughly follows the notation and language used in the standard textbook \cite{tcat} which we will periodically cite.  Fusion categories (over $\mathbb{C}$) are $\mathbb{C}$-linear, semisimple, rigid monoidal categories (with product $\otimes$, unit $\mathbbm{1}$, and duality $^\ast$), which have finitely many isomorphism classes of simple objects and simple monoidal unit.  The canonical examples of fusion categories are $\mathrm{Vec}_G$, the category of $G$-graded complex vector spaces, and $\mathrm{Rep}(G)$, the category of finite-dimensional complex representations of $G$, for any finite group $G$.  The concepts in the remainder of this paragraph are based on the underlying Grothendieck ring of a fusion category \cite[Chapter 3]{tcat}.  We denote the set of isomorphism classes of simple objects of a fusion category $\mathcal{C}$ by $\mathcal{O}(\mathcal{C})$.  The decomposition of $x\otimes y$ into simple objects for any $x,y\in\mathcal{O}(\mathcal{C})$ are referred to as the fusion rules of $\mathcal{C}$, and are encoded in the fusion matrices $N_x:=(N_{x,y}^z)_{y,z\in\mathcal{O}(\mathcal{C})}$ where $N_{x,y}^z:=\dim_\mathbb{C}\mathrm{Hom}_\mathcal{C}(x\otimes y,z)$.  The largest real eigenvalue of $N_x$ is known as the Frobenius-Perron dimension of $x$, or $\mathrm{FPdim}(x)$ for brevity, while the sum of $\mathrm{FPdim}(x)^2$ over all $x\in\mathcal{O}(\mathcal{C})$ will be denoted $\mathrm{FPdim}(\mathcal{C})$.  A simple object $x\in\mathcal{O}(\mathcal{C})$ is called invertible if $\mathrm{FPdim}(x)=1$ which implies $x\otimes y\in\mathcal{O}(\mathcal{C})$ for all $y\in\mathcal{O}(\mathcal{C})$.  When all $x\in\mathcal{O}(\mathcal{C})$ are invertible, we say $\mathcal{C}$ is pointed, and when $\mathrm{FPdim}(x)\in\mathbb{Z}$ for all simple $x$ in $\mathcal{C}$, we say $\mathcal{C}$ is integral.  But in general, $\mathcal{C}_\mathrm{pt}$ and $\mathcal{C}_\mathbb{Q}$ will be the maximal pointed and integral fusion subcategories of $\mathcal{C}$, respectively.

\begin{note}
We use the notation $\mathcal{C}_\mathbb{Q}$ in lieu of $\mathcal{C}_\mathbb{Z}$ since if $\mathbb{K}$ is any algebraic number field, the objects $x\in\mathcal{C}$ such that $\mathrm{FPdim}(x)\in\mathbb{K}$ form $\mathcal{C}_\mathbb{K}$, a fusion subcategory of $\mathcal{C}$ by \cite[Proposition 1.6]{2019arXiv191212260G}.  The fact that Frobenius-Perron dimensions are algebraic integers in these fields follows trivially from the definition.
\end{note}

The fusion categories $\mathrm{Vec}_G$ for finite groups $G$ are both integral and pointed, while the fusion categories $\mathrm{Rep}(G)$ are just integral unless $G$ is abelian.  Both of these families of examples have commutative fusion rules.  Moreso, for all objects $x,y$, one can choose natural isomorphisms $c_{x,y}:x\otimes y\to y\otimes x$ satisfying braid-like compatibilities; a fusion category along with a choice of natural isomorphisms $\{c_{x,y}\}_{x,y\in\mathcal{C}}$ satisfying the conditions in \cite[Definition 8.1.1]{tcat} is known as a braided fusion category.  If $\{c_{x,y}\}_{x,y\in\mathcal{C}}$ is a braiding on a fusion category $\mathcal{C}$, then $\{c_{y,x}^{-1}\}_{x,y\in\mathcal{C}}$ is also a braiding on $\mathcal{C}$; we denote this braided fusion category by $\mathcal{C}^\mathrm{rev}$.

\par If $\rho:G\to\mathrm{GL}(V)$ and $\sigma:G\to\mathrm{GL}(W)$ are finite-dimensional complex representations of a finite group $G$, then $\rho\otimes\sigma$ and $\sigma\otimes\rho$ are naturally isomorphic by simple transposition of bases of $V\otimes W$ and $W\otimes V$.  If we denote each transposition by $c_{\rho,\sigma}$, then $\mathrm{Rep}(G)$ along with $\{c_{\rho,\sigma}\}_{\rho,\sigma\in \mathrm{Rep}(G)}$ is a braided fusion category.  In this case $c_{\sigma,\rho}c_{\rho,\sigma}=\mathrm{id}_{\rho,\sigma}$ for all $\rho,\sigma\in\mathrm{Rep}(G)$ and we say that $\mathrm{Rep}(G)$ with this particular braiding is symmetrically braided \cite[Section 9.9]{tcat}.  But there are potentially many other symmetric braidings one can equip the fusion category $\mathrm{Rep}(G)$ with.  In particular, if $z\in Z(G)$ has order 2, then one can twist the standard symmetric braiding by the parity of the action of $z$ to produce another symmetric braiding \cite[Example 9.9.1]{tcat}.  To differentiate between these braided fusion categories, we will denote them by $\mathrm{Rep}(G,z)$, and $\mathrm{Rep}(G)$ will be reserved solely for the underlying fusion category.  For uniformity, $\mathrm{Rep}(G)$ equipped with the trivial symmetric braiding will be denoted $\mathrm{Rep}(G,e)$ and we refer to any braided fusion category braided equivalent to $\mathrm{Rep}(G,e)$ as Tannakian.

\begin{note}
It is crucial to emphasize the difference between an equivalence of fusion categories and an equivalence of braided fusion categories \cite[Definition 8.1.7]{tcat} which is a strictly stronger condition.  For example, \cite[Corollary 9.9.25]{tcat} states that if $\mathcal{C}$ is a symmetrically braided fusion category, then $\mathcal{C}$ is braided equivalent to $\mathrm{Rep}(G,z)$ for some finite group $G$ and $z\in Z(G)$ such that $z^2=e$.  But there are multitudes of examples of braided fusion categories $\mathcal{C}$ such that $\mathcal{C}$ is equivalent to $\mathrm{Rep}(G)$ as a fusion category for a finite group $G$ but $\mathcal{C}$ is not symmetrically braided (Section \ref{TY}).
\end{note}

Any fusion subcategory $\mathcal{D}\subset\mathcal{C}$ of a braided fusion category $\mathcal{C}$ is a braided fusion category with the braiding restricted to $\mathcal{D}$ from $\mathcal{C}$.  We define the relative centralizer of $\mathcal{D}$ in $\mathcal{C}$, denoted $C_\mathcal{C}(\mathcal{D})$, as the full subcategory of $x\in\mathcal{C}$ such that $c_{y,x}c_{x,y}=\mathrm{id}_{x\otimes y}$ for all $y\in\mathcal{C}$.  The special case of $C_\mathcal{C}(\mathcal{C})$ is known as the symmetric center of $\mathcal{C}$ as it is clearly a symmetrically braided fusion subcategory of $\mathcal{C}$.  In much of the existing literature, the symmetric center of $\mathcal{C}$ is simply denoted $\mathcal{C}'$ but we will avoid this notation in the remainder of the exposition as the apostrophe is an overburdened symbol in mathematics.  The condition of a braiding being symmetric can be restated as $\mathcal{O}(C_\mathcal{C}(\mathcal{C}))=\mathcal{O}(\mathcal{C})$, while we refer to braided fusion categories such that $\mathcal{O}(C_\mathcal{C}(\mathcal{C}))=\{\mathbbm{1}\}$ as \emph{nondegenerately braided}.  Of most importance in what follows is that if $\mathcal{D}$ is a braided fusion category and $\mathcal{C}\subset\mathcal{D}$ a fusion subcategory, then \cite[Theorem 8.21.5]{tcat}
\begin{equation}
\mathrm{FPdim}(\mathcal{C})\mathrm{FPdim}(C_\mathcal{D}(\mathcal{C}))=\mathrm{FPdim}(\mathcal{D})\mathrm{FPdim}(\mathcal{C}\cap C_\mathcal{D}(\mathcal{D})).\label{wanpisu}
\end{equation}


\subsection{Covers and extensions}\label{covers}

For a fixed braided fusion category, often one wants to consider categories containing the original which are larger but have more convenient properties.  In this \emph{laissez-faire} approach, the larger category will be called a \emph{cover} because the larger category has been loosely thrown on top without any specific information about how the smaller category is being contained.  When control is needed, the data of a cover along with specific instructions on how to connect the category to its cover will be called an \emph{extension}.

\begin{definition}
Let $\mathcal{C}$ be a braided fusion category.  If $\mathcal{D}$ is a braided fusion category and there exists a fully faithful braided tensor functor $\iota:\mathcal{C}\hookrightarrow\mathcal{D}$, then $\mathcal{D}$ is a \emph{cover} of $\mathcal{C}$, while the pair $(\mathcal{D},\iota)$ is an \emph{extension} of $\mathcal{C}$.
\end{definition}

Adjectives which apply to braided fusion categories apply to covers and extensions in a natural way.  For example, a \emph{nondegenerate cover} of $\mathcal{C}$ would be a nondegenerately braided fusion category $\mathcal{D}$ which is also a cover of $\mathcal{C}$. 

\begin{definition}\label{def}
Let $\mathcal{C}$ be a braided fusion category with extensions $(\mathcal{D},\iota)$ and $(\mathcal{E},\kappa)$.  A braided equivalence $F:\mathcal{D}\to\mathcal{E}$ is an \emph{equivalence of covers}.  If $F\circ\iota$ and $\kappa$ are naturally isomorphic as braided tensor functors, we say that $F$ is an \emph{equivalence of extensions}.
\end{definition}

Our description of extensions in Definition \ref{def} is essentially taken from \cite[Definition 4.1]{MR3613518} which emphasizes the importance of the embedding.  The original definition \cite[Conjecture 5.2]{mug1} in the setting of modular tensor categories aligns with our definition of a cover.  In a majority of the literature, there is a strong emphasis on studying nondegenerate covers or extensions $\mathcal{C}\subset\mathcal{D}$ which have the smallest possible Frobenius-Perron dimension, which by Equation (\ref{wanpisu}) is $\mathrm{FPdim}(\mathcal{D})=\mathrm{FPdim}(\mathcal{C})\mathrm{FPdim}(C_\mathcal{C}(\mathcal{C}))$.  This definition is equivalent to ensuring that no simple object is added in the cover which centralizes the original category.  In this way, such covers can be thought of as the optimal vehicle for applying the results of nondegenerately braided fusion categories to arbitrary braided fuson categories.

\begin{lemma}\label{previous}
Let $\mathcal{D}$ be a nondegenerate cover of a braided fusion category $\mathcal{C}$.  Then $\mathrm{FPdim}(\mathcal{D})=\mathrm{FPdim}(\mathcal{C})\mathrm{FPdim}(C_\mathcal{C}(\mathcal{C}))$ if and only if $C_\mathcal{D}(\mathcal{C})=C_\mathcal{C}(\mathcal{C})$.
\end{lemma}

\begin{proof}
It is clear that $C_\mathcal{C}(\mathcal{C})\subset C_\mathcal{D}(\mathcal{C})$ from definition.  But $\mathrm{FPdim}(C_\mathcal{D}(\mathcal{C}))\mathrm{FPdim}(\mathcal{C})=\mathrm{FPdim}(\mathcal{D})$ by Equation (\ref{wanpisu}) applied to $\mathcal{C}\subset\mathcal{D}$.  Therefore $\mathrm{FPdim}(\mathcal{D})=\mathrm{FPdim}(\mathcal{C})\mathrm{FPdim}(C_\mathcal{C}(\mathcal{C}))$ if and only if $\mathrm{FPdim}(C_\mathcal{C}(\mathcal{C}))=\mathrm{FPdim}(C_\mathcal{D}(\mathcal{C}))$, i.e.\ if and only if $C_\mathcal{C}(\mathcal{C})=C_\mathcal{D}(\mathcal{C})$.
\end{proof}

In the literature thus far, extensions satisfying the hypotheses of Lemma \ref{previous} have been referred to as minimal nondegenerate extensions or minimal modular extensions, with the addition of spherical structures (see Section \ref{modular}).  It has been shown that these extensions have additional algebraic structure in the unitary setting \cite[Theorem 1.1]{MR3613518} as well as physical meaning in the study of topological phases of matter \cite[Section 2]{MR3613518}.  But for many braided fusion cateogries the entire study of minimal nondegenerate extensions in this sense is vaccuous, as none exist.  We will retain this entrenched language for both covers and extensions with a hope that the notion of minimal extensions will eventually apply to all braided fusion categories.

%
%
%


\subsection{Graded extensions and modular data}\label{modular}

There are two structural tools for fusion and braided fusion categories that will be used frequently.

\par Firstly \cite[Section 4.14]{tcat}, each fusion category $\mathcal{C}$ (moreover, every fusion ring) possesses a universal grading, i.e.\ an additive decomposition of $\mathcal{C}$ by graded components corresponding to elements of a finite group $G$, such that the fusion rules of $\mathcal{C}$ respect the operation of $G$.  This grading is faithful in the sense that no graded component is empty, and we will refer to $G$ as the universal grading group of $\mathcal{C}$.  It follows from the definition of the universal grading that the trivial component, the fusion subcategory of $\mathcal{C}$ $\otimes$-generated by $x\otimes x^\ast$ for all $x\in\mathcal{O}(\mathcal{C})$, is a fusion subcategory which we denote by $\mathcal{C}_\mathrm{ad}$ and refer to as the adjoint subcategory \cite[Section 4.14]{tcat}.  In this way, if the universal grading group of $\mathcal{C}$ is a finite group $G$, we say $\mathcal{C}$ is a $G$-graded extension of $\mathcal{C}_\mathrm{ad}$.   If a fusion category $\mathcal{C}$ is faithfully graded by a finite group $G$, denote the graded components $\mathcal{C}_g$ for $g\in G$.  We will repeatedly use the fact that all graded components are the same dimension \cite[Theorem 3.5.2]{tcat}.  In particular, $\mathrm{FPdim}(\mathcal{C}_\mathrm{ad})=\sum_{x\in\mathcal{O}(\mathcal{C})\cap\mathcal{C}_g}\mathrm{FPdim}(x)^2$ for all $g\in G$ and moreover $\mathrm{FPdim}(\mathcal{C})=|G|\mathrm{FPdim}(\mathcal{C}_\mathrm{ad})$.

\par Secondly, if $\mathcal{C}$ is a fusion category and $\mathrm{FPdim}(\mathcal{C})\in\mathbb{Z}$ (i.e.\ $\mathcal{C}$ is weakly integral), then $\mathcal{C}$ possesses a canonical positive spherical structure \cite[Corollary 9.6.6]{tcat}.  In practice, we only need this structure to allow us to take traces of endomorphisms in $\mathcal{C}$, producing numerical constraints in our proofs.  The positivity of the canonical spherical structure corresponds to the fact that $\dim(x)$, the traces of the identity morphisms on $x\in\mathcal{O}(\mathcal{C})$, are precisely $\mathrm{FPdim}(x)$.  This implies $\mathrm{FPdim}(\mathcal{C})=\dim(\mathcal{C}):=\sum_{x\in\mathcal{O}(\mathcal{C})}\dim(x)^2$ hence in the setting of weakly integral fusion categories we will use $\mathrm{FPdim}$ and $\dim$ interchangeably, while for arbitrary fusion categories we will only use $\mathrm{FPdim}$.  We will consider nondegenerately braided fusion categories as modular tensor categories \cite[Definition 8.13.4]{tcat} equipped with their unique positive spherical structure, when it suits our purposes, in order to utilize their modular data \cite[Section 8.17]{tcat}.  Conveniently, the universal grading group of a modular tensor category $\mathcal{C}$ is canonically isomorphic to $\mathcal{O}(\mathcal{C}_\mathrm{pt})$ \cite[Theorem 6.3]{nilgelaki}.  The modular data of a modular tensor category is an invertible $|\mathcal{O}(\mathcal{C})|\times|\mathcal{O}(\mathcal{C})|$ matrix $S=(S_{x,y})_{x,y\in\mathcal{O}(\mathcal{C})}$ consisting of the traces of the double-braidings \cite[Section 8.13]{tcat}, and a diagonal matrix with diagonal elements $(\theta_x)_{x\in\mathcal{O}(\mathcal{C})}$ where $\theta_x$ is the trace of the ribbon structure $\mathcal{C}$.  The matrix $S$ is unitary \cite[Proposition 2.12]{ENO}; we will refer to this fact as orthogonality relations between simple objects of $\mathcal{C}$.  Specifically, if $x,y\in\mathcal{O}(\mathcal{C})$, then $\sum_{z\in\mathcal{O}(\mathcal{C})}S_{x,z}\overline{S_{y,z}}=\dim(\mathcal{C})$ if $x\cong y$, and is zero otherwise, using $\overline{\alpha}$ as the complex conjugate of $\alpha\in\mathbb{C}$.  It is known that $\theta_x$ are roots of unity for all $x\in\mathcal{O}(\mathcal{C})$ \cite[Corollary 8.18.2]{tcat} and we will refer to the order of the $T$-matrix as the conductor of $\mathcal{C}$.  Two formulas that will be used in our proofs are the Verlinde formula \cite[Corollary 8.14.4]{tcat}, which states that the fusion rules of a modular tensor category are given by the $S$-matrix via
\begin{equation}
\dim(\mathcal{C})N_{y,z}^w=\sum_{x\in\mathcal{O}(\mathcal{C})}\dfrac{S_{x,y}S_{x,z}\overline{S_{x,w}}}{\dim(x)}
\end{equation}
for all $w,y,z\in\mathcal{O}(\mathcal{C})$, and the balancing equation \cite[Proposition 8.13.8]{tcat}, which states that even without nondegeneracy of the braiding,
\begin{equation}
S_{x,y}=\theta_x^{-1}\theta_y^{-1}\sum_{z\in\mathcal{O}(\mathcal{C})}N_{x,y}^z\theta_z\dim(z).
\end{equation}
The last benefit of the $S$-matrix is that it allows a numerical test for two objects $x,y$ in a spherical braided fusion category $\mathcal{C}$ to centralize one another, i.e.\ $c_{y,x}c_{x,y}=\mathrm{id}_{x\otimes y}$.  By taking the trace of these endomorphisms we can see $x,y$ centralizing one another implies $S_{x,y}=\dim(x)\dim(y)$, and the converse is also true \cite[Proposition 2.5]{mug1}.

\subsection{Extraspecial $p$-groups and their character rings}\label{extraspecial}

Let $p$ be a prime integer.  Recall that a finite group $G$ is a $p$-group if the order of $G$ is a power of $p$.  A $p$-group $G$ is \emph{extraspecial} if $|Z(G)|=p$ and $G/Z(G)$ is a non-trivial elementary abelian $p$-group, i.e.\ isomorphic to $(\mathbb{Z}/p\mathbb{Z})^{\oplus n}$ for some $n\in\mathbb{Z}_{\geq1}$.  For each prime integer $p$ and positive integer $n$, there exist exactly two isomorphism classes of extraspecial $p$-groups of order $p^{2n+1}$ which are traditionally denoted $p^{1+2n}_\pm$.  There do not exist extraspecial $p$-groups whose order is an even power of $p$.  All extraspecial $p$-groups can be constructed as central products of the two extraspecial $p$-groups of order $p^3$.  For example, the extraspecial $2$-groups of order 8 are the dihedral group $D_4=2^{1+2}_+$ and the quaternion group $Q_8=2^{1+2}_-$.  Therefore, all extraspecial $2$-groups are just central products of various copies of $D_4$ and $Q_8$.
\par The representation theory of extraspecial $p$-groups is straightforward as $p^{2n}$ isomorphism classes of irreducible representations of $p^{1+2n}_\pm$ are one-dimensional, corresponding to the elementary abelian $p$-group underlying $p^{1+2n}_\pm$.  We will abuse notation and denote these classes by $g$ for $g\in (\mathbb{Z}/p\mathbb{Z})^{\oplus 2n}$ whose fusion rules are simply those of the linear characters of $(\mathbb{Z}/p\mathbb{Z})^{\oplus 2n}$.  The remaining $p-1$ isomorphism classes of irreducible representations are faithful of dimension $p^n$, and are distinguished by their values on $Z(p^{1+2n}_\pm)$.  Denote these classes by $x_g$ for $g\in(\mathbb{Z}/p\mathbb{Z})^\times$.  The fusion rules involving $x_h$ are then
\begin{align}\label{too}
&g\otimes x_h\cong x_h\otimes g\cong x_h,\,\,\,\text{ for all }g\in (\mathbb{Z}/p\mathbb{Z})^{\oplus 2n},h\in(\mathbb{Z}/p\mathbb{Z})^\times  \\
&x_g\otimes x_h\cong p^nx_{gh}\,\,\,\text{ if }gh\neq e,\text{ and } \\
&x_g\otimes x_{g^{-1}}\cong\bigoplus_{h\in (\mathbb{Z}/p\mathbb{Z})^{\oplus 2n}} h.
\end{align}
Extraspecial $p$-groups are characterized by the degrees of their characters.  This result can be found in standard textbooks on the character theory of finite groups such as \cite[Proposition 7.7]{MR1645304}.  For our purposes, we will only need a trivial corollary of this fact.

\begin{lemma}\label{extraspeciallem}
Let $R$ be the character ring of an extraspecial $p$-group.  If $G$ is a finite group whose character ring is isomorphic to $R$, then $G$ is an extraspecial $p$-group.
\end{lemma}



\subsection{Braided Tambara-Yamagami}\label{TY}

\par Tambara-Yamagami fusion categories are $\mathbbm{Z}/2\mathbb{Z}$-graded extensions (see Section \ref{modular}) of pointed fusion categories whose non-trivial graded component has exactly one isomorphism class of simple objects.  We will denote the isomorphism classes of invertible objects by $g\in G$, a finite group, and the isomorphism class of the noninvertible simple object by $x$.  The fusion rules for invertible elements follow the group operation of $G$ while fusion rules with $x$ must be $g\otimes x\cong x\otimes g\cong x$ and $x\otimes x\cong\bigoplus_{g\in G}g$.  For example, $\mathrm{Rep}(2_\pm^{1+2n})$ for $n\in\mathbb{Z}_{\geq1}$ are Tambara-Yamagami fusion categories whose adjoint subcategories have the fusion rules of an elementary abelian $2$-group.  It was shown in \cite{MR1659954} that a Tambara-Yamagami fusion category (over $\mathbb{C}$) is characterized by the group of invertible objects, $G$, which must be abelian, a nondegenerate bicharacter $\chi:G\times G\to\mathbb{C}^\times$ such that $\chi(g,h)=\chi(h,g)$ for all $g,h\in G$, and an extension sign \cite[Example 9.4]{MR2677836} which is usually recorded as $\tau$, a square root of $|G|^{-1}$ as this value appears in many computations.  It is reasonable to use the notation $\mathcal{C}(\chi,\tau)$ for such a category since a description of $\chi$ includes a description of the finite abelian group $G$.  This is the original notation of D.\ Tambara and S.\ Yamagami \cite[Definition 3.1]{MR1659954}.

\par It has been proven \cite[Theorem 1.2(1)]{siehler} that Tambara-Yamagami fusion categories possess a braiding if and only if $G$ is an elementary abelian 2-group, i.e.\ $G=E_n:=(\mathbb{Z}/2\mathbb{Z})^{\oplus n}$ for some $n\in\mathbb{Z}_{\geq1}$.   Up to isometry, there is a unique nondegenerate bicharacter $\chi^1$ on $E_n$ such that $\chi^1(g,h)=\chi^1(h,g)$ for all $g\in E_n$ when $n$ is odd and exactly two, $\chi^0,\chi^1$, up to isomorphism when $n$ is even \cite[Section 5]{MR156890}.  Specifically, $\chi^0(g,g)=1$ and $\chi^1(g,g)=-1$ for all generators $g\in E_n$.  When no specific braiding needs to be defined, it will suffice to denote all braided Tambara-Yamagami fusion categories as $\mathcal{C}(\chi^k_n,\tau)$ where $k=1,2$ and $n$ indicates $E_n$ is the underlying group of invertible objects.

\par For fixed $n\in\mathbb{Z}_{\geq1}$, it was further proven that there exist at most $2^{n+1}$ inequivalent braidings for each $\mathcal{C}(\chi_n^k,\tau)$, indexed by choices of signs $(\delta_1,\ldots,\delta_n,\epsilon)$ \cite[Theorem 1.2(2)]{siehler}.  Many of these braidings are equivalent.  A more precise statement can be found in \cite[Corollary 4.10]{galindo2020trivializing} and elsewhere.  Here the data $(\delta_1,\ldots,\delta_n,\epsilon)$ is replaced with $(q,\alpha)$ where $q:E_n\to\mathbb{C}^\times$ is a quadratic form such that $\chi_n^k(g,h)=q(g)q(h)q(g+h)^{-1}$ for all $g,h\in E_n$, and $\alpha\in\mathbb{C}$ is a chosen square root of $\tau\sum_{g\in E_n}q(g)$.  To translate between the two sets of data, $q(g)=\sigma_1(g)$ for all $g\in E_n$ in the notation of \cite[Section 2.3]{siehler}, while $\alpha=\sigma_3(e)$.  Two braidings $(q,\alpha)$, $(q',\alpha')$ on a fixed fusion category $\mathcal{C}(\chi,\tau)$ with $\chi:E_n\times E_n\to\mathbb{C}^\times$ are equivalent if and only if there exists a group automorphism $f\in\mathrm{Aut}(E_n)$ such that $q'(f(g))=q(g)$ for all $g\in E_n$ and $\alpha=\alpha'$ \cite[Corollary 4.10]{galindo2020trivializing}.  This implies that $f$ satisfies $\chi(g,h)=\chi(f(g),f(h))$ for all $g,h\in E_n$.

\begin{example}\label{ice}
The braided equivalence classes of $\mathcal{C}(\chi_1^1,\tau)$ are given by the following sets of data in Figure \ref{fig:ising}, displayed in both the notation of \cite[Section 2.3]{siehler} and \cite[Section 4.3]{galindo2020trivializing}, where we have indexed the categories by the primitive 16th root of unity $\alpha$.  We use the label $\mathcal{I}$ as they have traditionally been referred to as Ising categories \cite[Appendix B]{DGNO}.
\begin{figure}[H]
\centering
\begin{equation*}
\begin{array}{|c|c|cc|ccc|}
\hline  & \tau & \delta &\epsilon & q(e) & q(g) & \alpha\\\hline
\mathcal{I}_1& 1/\sqrt{2} & 1 & 1 & 1 & i & \zeta_{16} \\
\mathcal{I}_3& -1/\sqrt{2} & -1 & 1 & 1 & -i & \zeta_{16}^3 \\
\mathcal{I}_5& -1/\sqrt{2} & 1 & -1 & 1 & i & \zeta_{16}^5\\
\mathcal{I}_7& 1/\sqrt{2} & -1 & -1 & 1 & -i & \zeta_{16}^7\\
\mathcal{I}_9& 1/\sqrt{2} & 1 & -1 & 1 & i & \zeta_{16}^9 \\
\mathcal{I}_{11}& -1/\sqrt{2} & -1 & -1 & 1 & -i & \zeta_{16}^{11}\\
\mathcal{I}_{13}& -1/\sqrt{2} & 1 & 1 & 1 & i & \zeta_{16}^{13}\\
\mathcal{I}_{15}& 1/\sqrt{2} & -1 & 1 & 1 & -i & \zeta_{16}^{15} \\
\hline
\end{array}
\end{equation*}
    \caption{$\chi_1^1$ braidings}%
    \label{fig:ising}%
\end{figure}
\end{example}

\begin{example}\label{kequals2}
Order the elements of $E_2=(\mathbb{Z}/2\mathbb{Z})^{\oplus2}$ as $e,g_1,g_2,g_1+g_2$.  For reference, the nondegenerate bilinear forms $\chi_2^k:E_2\times E_2\to\mathbb{C}^\times$ are given by 
\begin{equation}
\chi_2^0:\,\,\,\left[\begin{array}{cccc}
 1& 1 & 1 & 1 \\
 1& 1 & -1 & -1 \\
 1& -1 & 1 & -1 \\
 1& -1 & -1 & 1.
\end{array}\right],\qquad\text{ and }\qquad
\chi_2^1:\,\,\,\left[\begin{array}{cccc}
 1& 1 & 1 & 1 \\
 1& -1 & 1 & -1 \\
 1& 1 & -1 & -1 \\
 1& -1 & -1 & 1
\end{array}\right].
\end{equation}
There exist at most 8 braidings for each category which are defined by three sign choices in \cite{siehler}: $\delta_1$, $\delta_2$, and $\epsilon$.  The choice of $\epsilon$ will always result in 2 inequivalent braidings by \cite[Corollary 4.10(c)]{galindo2020trivializing} when other parameters are held constant, so it remains to determine whether any assignments of $\delta_1,\delta_2$ are equivalent.  Commutation of group elements $g_1,g_2$ with the noninvertible simple $x$ acts as multiplication by $q(g_j)=\sigma_1(g_j)=\delta_j\sqrt{\chi(g_j,g_j)}$ for $j=1,2$, and as multiplication by
\begin{equation}
q(g_1+g_2)=\sigma_1(g_1+g_2)=\delta_1\delta_2\sqrt{\chi(g_1,g_1)}\sqrt{\chi(g_2,g_2)}\chi(g_1,g_2)=-\delta_0\delta_1
\end{equation}
by the commutation of $g_1+g_2$ with $x$.  Lastly, since $x\otimes x=e\oplus g_1\oplus g_2\oplus(g_1+g_2)$, we can describe the braiding of $x$ with itself on the component corresponding to $e\in E_2$ as multiplication by
\begin{equation}\label{sigma3}
\alpha=\sigma_3(e)=\epsilon\sqrt{\tau\left(1+\delta_1\sqrt{\chi(g_1,g_1)}+\delta_2\sqrt{\chi(g_2,g_2)}-\delta_1\delta_2\right)}.
\end{equation}
We have $\mathrm{Aut}(E_2)\cong S_3$, but for $\chi_2^1$, only the permutation $g_1\leftrightarrow g_2$ preserves $q$ as $q(g_1+g_2)\in\mathbb{R}$ while $q(g_1),q(g_2)\not\in\mathbb{R}$, leaving 6 inequivalent braidings on each $\mathcal{C}(\chi_2^1,\tau)$.  For $\chi_2^0$, $q(g)=\pm1$ for all $g\in E_2$ and any two braidings are equivalent if the images of their quadratic forms $q$ have the same number of 1's and $-1$'s, and $\alpha=\sigma_3(e)$ corresponding to each is equal.  Therefore, for a fixed fusion category $\mathcal{C}(\chi_2^0,\tau)$ there are $4$ inequivalent braidings for a total of $16$ braided equivalence classes of $E_2$ Tambara-Yamagami braided fusion categories.

\par One can verify with the above formulas that the collections of braiding data in Figure \ref{fig:one}, displayed in both the notation of \cite[Section 2.3]{siehler} and \cite[Section 4.3]{galindo2020trivializing}, correspond to the 4 symmetrically braided $E_2$ Tambara-Yamagami fusion categories, where $z\in Z(G)$ for $G=D_4,Q_8$ is the unique nontrivial central element.  The observation that $\mathcal{C}(\chi_2^0,1/2)$ is equivalent to $\mathrm{Rep}(D_4)$ and $\mathcal{C}(\chi_2^0,-1/2)$ is equivalent to $\mathrm{Rep}(Q_8)$ as fusion categories was made in \cite[Section 4]{MR1659954}.  The remaining 4 equivalence classes of braided fusion categories $\mathcal{C}(\chi_2^0,\tau)$ in Figure \ref{fig:onetwo}, and 12 equivalence classes of braided fusion categories $\mathcal{C}(\chi_2^1,\tau)$ in Figure \ref{fig:onetwothree}, are not symmetrically braided.  We label these braided fusion categories by their realizations which are discussed in Section \ref{interesting}.

\begin{figure}[H]
\centering
\begin{equation*}
\begin{array}{|c|c|ccc|ccccc|}
\hline  & \tau & \delta_1 & \delta_2 &\epsilon & q(e) & q(g_1) & q(g_2) & q(g_1+g_2) & \alpha\\\hline
\mathrm{Rep}(D_4,e) & 1/2 & 1 & 1 & 1 & 1 &  1 & 1 & -1 & 1 \\
\mathrm{Rep}(D_4,z) & 1/2 & 1 & 1 & -1 & 1 & 1 & 1 & -1 & -1  \\
\mathrm{Rep}(Q_8,e) & -1/2 & -1 & -1 & 1 & 1 & -1 & -1 & -1 & 1  \\
\mathrm{Rep}(Q_8,z) & -1/2 & -1 & -1 & -1 & 1 & -1 & -1 & -1 & -1  \\\hline
\end{array}
\end{equation*}
    \caption{Symmetric $\chi_2^0$ braidings}%
    \label{fig:one}%
\end{figure}

\begin{figure}[H]
\centering
\begin{equation*}
\begin{array}{|c|c|ccc|ccccc|}
\hline  & \tau & \delta_1 & \delta_2 &\epsilon & q(e) & q(g_1) & q(g_2) & q(g_1+g_2) & \alpha\\\hline
\mathcal{K} & 1/2 & -1 & -1 & 1 & 1 & -1 & -1 & -1 & i \\
\mathcal{K}^\mathrm{rev} & 1/2 & -1 & -1 & -1 & 1 & -1 & -1 & -1 & -i  \\
\mathcal{Z}(\mathrm{Vec}^\gamma_{Q_8})_\mathrm{ad} & -1/2 & 1 & 1 & 1 & 1 & 1 & 1 & -1 & i  \\
\mathcal{Z}(\mathrm{Vec}^\gamma_{Q_8})^\mathrm{rev}_\mathrm{ad} & -1/2 & 1 & 1 & -1 & 1 & 1 & 1 & -1 & -i  \\
\hline
\end{array}
\end{equation*}
    \caption{Nonsymmetric $\chi_2^0$ braidings}%
    \label{fig:onetwo}%
\end{figure}

\begin{figure}[H]
\centering
\begin{equation*}
\begin{array}{|c|c|ccc|ccccc|}
\hline  & \tau & \delta_1 & \delta_2 &\epsilon & q(e) & q(g_1) & q(g_2) & q(g_1+g_2) & \alpha\\\hline
(\mathcal{I}_1\boxtimes\mathcal{I}_1)_\mathbb{Q} & 1/2 & 1 & 1 & 1 & 1 & i & i & -1 & \zeta_8 \\
(\mathcal{I}_5\boxtimes\mathcal{I}_5)_\mathbb{Q} & 1/2 & 1 & 1 & -1 & 1 & i & i & -1 & \zeta^5_8  \\
(\mathcal{I}_{\zeta_1}\boxtimes\mathcal{I}_{15})_\mathbb{Q} & 1/2 & 1 & -1 & 1 & 1 & i & -i & 1 & 1  \\
(\mathcal{I}_1\boxtimes\mathcal{I}_7)_\mathbb{Q} & 1/2 & 1 & -1 & -1 & 1 & i & -i & 1 &  -1 \\
(\mathcal{I}_7\boxtimes\mathcal{I}_7)_\mathbb{Q}  & 1/2 & -1 & -1 & 1 & 1 & -i & -i & -1 & \zeta_8^7 \\
(\mathcal{I}_3\boxtimes\mathcal{I}_3)_\mathbb{Q} & 1/2 & -1 & -1 & -1 & 1 & -i & -i & -1 & \zeta_8^3  \\
(\mathcal{I}_1\boxtimes\mathcal{I}_{13})_\mathbb{Q} & -1/2 & 1 & 1 & 1 & 1 & i & i & -1 & \zeta_8^7 \\
(\mathcal{I}_1\boxtimes\mathcal{I}_5)_\mathbb{Q} & -1/2 & 1 & 1 & -1 & 1 & i & i & -1 &  \zeta_8^3 \\
(\mathcal{I}_1\boxtimes\mathcal{I}_3)_\mathbb{Q}& -1/2 & 1 & -1 & 1 & 1 & i & -i & 1 & i \\
(\mathcal{I}_1\boxtimes\mathcal{I}_{11})_\mathbb{Q} & -1/2 & 1 & -1 & -1 & 1 & i & -i & 1 & -i \\
(\mathcal{I}_3\boxtimes\mathcal{I}_{15})_\mathbb{Q} & -1/2 & -1 & -1 & 1 & 1 & -i & -i & -1 & \zeta_8 \\
(\mathcal{I}_3\boxtimes\mathcal{I}_7)_\mathbb{Q} & -1/2 & -1 & -1 & -1 & 1 & -i & -i & -1 & \zeta_8^5 \\
\hline
\end{array}
\end{equation*}
    \caption{$\chi_2^1$ braidings}%
    \label{fig:onetwothree}%
\end{figure}

\end{example}

\subsection{Doubles and algebras}\label{doubles}

Lastly, we outline the center, or double construction, for fusion categories and its relation to commutative algebras.  If $\mathcal{C}$ is a fusion category, then $\mathcal{Z}(\mathcal{C})$, the center or double of $\mathcal{C}$ has objects in correspondence with pairs $(x,\{\rho_y\}_{y\in\mathcal{C}})$ where $x\in\mathcal{C}$ and $\rho_y:x\otimes y\to y\otimes x$ is a natural isomorphism for all $y\in\mathcal{C}$ satisfying the coherence diagram in \cite[Definition 7.13.1]{tcat} where one can find the definition of morphisms in this category, its monoidal product, etc.  The goal of this construction is to create a nondegenerately braided fusion category from $\mathcal{C}$ in a uniform manner, which often ends up being quite unwieldly.  The namesake of the double comes from the fact that the representation category of the quantum double of a Hopf algebra $H$ is precisely the double $\mathcal{Z}(\mathrm{Rep}(H))$ \cite[Proposition 7.14.6]{tcat}.  Let $G$ be a finite group and $\omega\in H^3(G,\mathbb{C}^\times)$ a cohomological twisting of the associativity of $\mathrm{Vec}_G$, producing the fusion category of twisted $G$-graded vector spaces $\mathrm{Vec}_G^\omega$ \cite[Example 2.3.8]{tcat}.  The center $\mathcal{Z}(\mathrm{Vec}_G^\omega)$ is known as the twisted double of $G$ whose modular data has been known for quite some time \cite[Sections 2.2 \& 5.2]{gannon1996}.

\par Alternatively, twisted doubles of finite groups can be characterized by possessing maximal Tannakian fusion subcategories which we will understand through the following construction.  In general, if $\mathcal{C}$ is a braided fusion category and $\mathrm{Rep}(G)\subset\mathcal{C}$ is a Tannakian fusion subcategory for a finite group $G$, then $R:=\oplus_{g\in G}g$ has the structure of a connected \'etale algebra in $\mathcal{C}$ \cite[Section 3]{DMNO}; we will refer to this algebra as the regular algebra of the Tannakian subcategory.  The category of local $R$-modules $\mathcal{C}_R^0$ is then a braided fusion category which is nondegenerately braided if and only if $\mathcal{C}$ is \cite[Corollary 3.30]{DMNO}.  Moreover $\mathcal{C}_R^0$ inherits the spherical structure of $\mathcal{C}$ so the passage $\mathcal{C}\to\mathcal{C}_R^0$ sends modular tensor categories to modular tensor categories.  When $R$ is the regular algebra of a Tannakian fusion subcategory $\mathcal{E}\subset\mathcal{C}$, then the simple objects of $\mathcal{C}_R^0$ are summands of the free $R$-modules $R\otimes x$ for $x\in\mathcal{O}(C_\mathcal{C}(\mathcal{E}))$ \cite[Example 3.14]{DMNO}. 

\par We will use regular algebras of Tannakian subcategories for numerical arguments as well, but one important application is to prove a reconstruction theorem for twisted doubles of finite groups from Lagrangian subcategories, i.e.\ Tannakian fusion subcategories $\mathcal{E}$ of braided fusion categories $\mathcal{C}$ such that $\mathrm{FPdim}(\mathcal{E})^2=\mathrm{FPdim}(\mathcal{C})$.  Theorem 4.5 of \cite{drinfeld2007grouptheoretical} states that a modular tensor category $\mathcal{C}$ is braided equivalent to a twisted double of a finite group $G$ if and only if there exists a Lagrangian subcategory $\mathrm{Rep}(G)\subset\mathcal{C}$.


\section{Nondegenerate covers and extensions of $\mathcal{C}(\chi_2^k,\tau)$}\label{interesting}

Here we describe minimal nondegenerate covers (Section \ref{covers}) for nonsymmetrically braided $E_2$ Tambara-Yamagami fusion categories when they exist, while $\mathcal{Z}(\mathcal{C})$ is known to be a minimal nondegenerate cover for all symmetrically braided fusion categories $\mathcal{C}$.  The main objective of this section is to provide an independent proof that there exists a braided fusion category of Frobenius-Perron dimension 8 whose symmetric center has Frobenius-Perron dimension 4, which does not possess a nondegenerate cover of Frobenius-Perron dimension of 32 \cite[Proposition 4.11]{galindo2017categorical}.  In fact, there are 2 such braided fusion categories, equivalent as fusion categories, with reverse braidings, denoted $\mathcal{K}$ and $\mathcal{K}^\mathrm{rev}$ in Figure \ref{fig:onetwo}.  We do so in a way that generalizes to an infinite family of examples which have not appeared in the literature thus far in Section \ref{jen}, and also illustrates the structure between some minimal nondegenerate covers (Example \ref{structure}).


\subsection{Nonsymmetric $\chi_2^0$ braidings}\label{firstexample}

\par Let $\mathcal{C}$ be one of the 4 nonsymmetrically braided fusion categories $\mathcal{C}(\chi^0_2,\tau)$ (Figure \ref{fig:onetwo}).  Note that the symmetric center $C_\mathcal{C}(\mathcal{C})=\mathcal{C}_\mathrm{pt}$ is Tannakian, using the formulas of \cite[Section 2.3(2)]{siehler}.  Let $\mathcal{D}$ be a nondegenerate cover of $\mathcal{C}$ with $\mathrm{FPdim}(\mathcal{D})=32$.  As $\mathcal{D}$ is weakly integral, we may assume $\mathcal{D}$ is a modular tensor category equipped with its unique positive spherical structure. We will refer to $\mathcal{O}(\mathcal{C}_\mathrm{pt})$ by its structure as an abelian group, $E_2=(\mathbb{Z}/2\mathbb{Z})^2$.

\par We will first demonstrate that $\mathcal{D}_\mathrm{pt}=\mathcal{C}_\mathrm{pt}$ and $\mathcal{D}_\mathrm{ad}=\mathcal{C}$.  To this end, since $\mathcal{D}$ is weakly integral, then \cite[Lemma 1.1]{2019arXiv191212260G} implies $\mathcal{D}_\mathrm{ad}$ is integral, of dimension $4$ or $8$ since $\dim(\mathcal{D})=\dim(\mathcal{D}_\mathrm{pt})\dim(\mathcal{D}_\mathrm{ad})$ by \cite[Corollaries 8.21.7 \& 8.22.8]{tcat}.  Recall that $\mathcal{D}$ is nilpotent, as its dimension is a prime power \cite[Example 4.5]{nilgelaki}.  Hence in the former case, every $y\in\mathcal{O}(\mathcal{D})$ satisfies $\dim(y)^2\in\{1,2,4\}$ as $\dim(y)^2$ must divide $\dim(\mathcal{D}_\mathrm{ad})$ \cite[Theorem 5.2.]{nilgelaki}.  If there exists simple $y\in\mathcal{D}$ with $\dim(y)^2=2$ then there exists a unique simple object $g\in E_2$ with $g\otimes y\cong y$: the nontrivial summand of $y\otimes y^\ast$.   The other two nontrivial $h_1,h_2\in E_2$ must then satisfy $h_j \otimes y\cong z$ where $z$ is another simple object in the same universally graded component as $y$ with $\dim(z)^2=2$.  Morever $y\otimes z\cong h_1\oplus h_2\in\mathcal{D}_\mathrm{ad}$.  Therefore $y\otimes y\cong\mathbbm{1}\oplus g$ (and $z\otimes z\cong\mathbbm{1}\oplus g$), implying $y\cong y^\ast$.  This forces $y$ to $\otimes$-generate a braided fusion category of dimension 4 which is not pointed.  Such a category is an Ising modular tensor category, whose pointed subcategory is not Tannakian \cite[Lemma B.18]{DGNO}.  But $\mathcal{C}_\mathrm{pt}$ is Tannakian so this is cannot be the case.  Hence every simple object is invertible or dimension 2 and thus lies in its own universally graded component.  This implies $g\otimes x\cong x$ for any $x$ such that $\dim(x)=2$, hence $S_{g,x}=2$ by the balancing equation and thus $x$ is centralized by $\mathcal{D}_\mathrm{ad}$.  This is a contradiction since $C_\mathcal{D}(\mathcal{D}_\mathrm{ad})=\mathcal{D}_\mathrm{pt}$.  We may then conclude that $\dim(\mathcal{D}_\mathrm{ad})=8$, hence $\mathcal{D}_\mathrm{pt}=\mathcal{C}_\mathrm{pt}$, and lastly $\mathcal{D}_\mathrm{ad}=\mathcal{C}$.

\par We may now reveal the $S$-matrix of $\mathcal{D}$ (Section \ref{modular}).  The balancing equation implies $S_{x,x}=S_{x,x^\ast}=4\theta_x^{-2}$ since $\theta_g=1$ for all $g\in E_2$, while $S_{x,g}=2$ for all $g\in E_2$ since $x$ centralizes $g$.  From the formulas of \cite[Section 3.7]{siehler} and \cite[Section 2.3(2)]{siehler}, $\theta_x=\pm i$ since $\mathcal{C}$ is not symmetric.  Moreover $S_{x,x}=-4$.  The orthogonality relation (Section \ref{modular}) of $x$ with itself is
\begin{equation}
32=\sum_{z\in\mathcal{O}(\mathcal{D})}|S_{x,z}|^2=\sum_{z\in\mathcal{O}(\mathcal{C})}|S_{x,z}|^2+\sum_{z\in\mathcal{O}(\mathcal{D})\setminus\mathcal{O}(\mathcal{C})}|S_{x,z}|^2=4\cdot 2^2+4^2+\sum_{z\in\mathcal{O}(\mathcal{D})\setminus\mathcal{O}(\mathcal{C})}|S_{x,z}|^2.
\end{equation}
Therefore $\sum_{z\in\mathcal{O}(\mathcal{D})\setminus\mathcal{O}(\mathcal{C})}|S_{x,z}|^2=0$ and moreover $S_{x,z}=0$ for all $z\in\mathcal{O}(\mathcal{D})\setminus\mathcal{O}(\mathcal{C})$.  Verlinde formula then implies for $y\in\mathcal{O}(\mathcal{D})\setminus\mathcal{O}(\mathcal{C})$,
\begin{equation}
\dim(\mathcal{D})N_{x,y}^y=\sum_{z\in E_2}\dfrac{S_{x,z}S_{y,z}\overline{S_{y,z}}}{\dim(z)}=\sum_{z\in E_2}S_{x,z}|S_{y,z}|^2,
\end{equation}
since for every $z\in\mathcal{O}(\mathcal{D})\setminus E_2$, we have shown either $S_{x,z}=0$ or $S_{y,z}=0$.  Moreover, we have computed the remaning $S$-matrix entries above, implying
\begin{equation}
N^x_{y,y^\ast}=N_{x,y}^y=\dfrac{1}{32}\sum_{z\in E_2}S_{x,z}|S_{y,z}|^2
=\dfrac{1}{32}\sum_{z\in E_2}2\cdot\dim(y)^2=\dfrac{\dim(y)^2}{4}.
\end{equation}
As this and $\dim(y)$ must be integers, and $\dim(y)^2$ divides $\dim(\mathcal{D}_\mathrm{ad})=8$, then $\dim(y)^2=4$ or $\dim(y)^2=8$.  But if $\dim(y)^2=8$ then $g\otimes y\cong y$ for all $g\in E_2$ since $y$ is the unique simple object in its universally graded component, hence $y$ is centralized by $E_2$, i.e.\ $y\in\mathcal{C}_\mathrm{ad}$, a contradiction.  Therefore $\dim(y)=2$ and thus each nontrivially graded component contains exactly 2 simple objects of dimension 2, which are permuted transitively by $E_2$.  Moreover $\mathrm{rank}(\mathcal{D})=11$.

\par To compute the remaining $S$-matrix entries, let $y\in\mathcal{O}(\mathcal{D})\setminus\mathcal{O}(\mathcal{C})$ with $y'$ the only other simple object in the orbit of $y$ under the $\otimes$-action of $E_2$.  The orthogonality relation of $y$ with $x$ is
\begin{equation}
0=\sum_{z\in\mathcal{O}(\mathcal{D})}S_{y,z}S_{x,z}=\sum_{z\in E_2}S_{y,z}S_{x,z}=\sum_{z\in E_2}\dfrac{\theta_{y\otimes z}}{\theta_y}\dim(y)\dim(x)=\dfrac{8}{\theta_y}(\theta_y+\theta_{y'}).
\end{equation}
Therefore $\theta_y=-\theta_{y'}$.  In particular, $y\cong y^\ast$ for all $y\in\mathcal{O}(\mathcal{D})\setminus\mathcal{O}(\mathcal{C})$ as duality is a permutation of each universally graded component since $E_2$ has exponent 2, along with the fact that $\theta_{y^\ast}=\theta_y$ in complete generality.  The fact that $\theta_y=-\theta_{y'}$ also completes the $S$-matrix columns for $g\in E_2$ since $S_{g,y}=S_{g,y}=2$ if $g\otimes y=y$ and $S_{g,y}'=-2$ if $g\otimes y\cong y'$.  The balancing equation then implies
\begin{equation}
S_{y,y}=\theta_y^{-2}\left(\sum_{z\in E_2}N_{y,y}^z\theta_z+2\theta_x+\sum_{z\in\mathcal{O}(\mathcal{D})\setminus\mathcal{O}(\mathcal{C})}2N_{y,y}^z\theta_z\right)=2\theta_y^{-2}(1+\theta_x).
\end{equation}
It was already determined that $\theta_x=\pm i$, so $(1/\sqrt{2})(1+\theta_x)=\zeta_8^{-i\theta_x}$ where $\zeta_8:=\exp(2\pi i/8)$.  Thus $S_{y,y}=2\sqrt{2}\cdot\theta_y^{-2}\zeta_8^{-i\theta_x}$.  But $y\cong y^\ast$, hence $S_{y,y}\in\mathbb{R}$ and therefore $\theta_y$ is a primitive 16th root of unity for all $y\in\mathcal{O}(\mathcal{D})\setminus\mathcal{O}(\mathcal{C})$.  Moreover, we may define the signs $\varepsilon_y:=\theta_y^{-2}\zeta_8^{-i\theta_x}$ so that $S_{y,y}=\varepsilon_y2\sqrt{2}$.  Furthermore, if $y\not\cong y'$ are in the same nontrivial graded component with $g\otimes y\cong y'$ for some $g\in G$, then by \cite[Proposition 8.13.10]{tcat},
\begin{equation}
S_{y,y'}=S_{y,g\otimes y}=\dfrac{1}{2}S_{y,g}S_{y,y}=\dfrac{\theta_{g\otimes y}}{\theta_y}S_{y,y}=-S_{y,y}.
\end{equation}
Orthogonality of $y$ with itself is then
\begin{equation}
32=\sum_{z\in\mathcal{O}(\mathcal{D})}|S_{y,z}|^2=4\cdot2^2+2\cdot(2\sqrt{2})^2+\sum_{z\in\mathcal{O}(\mathcal{D})\setminus(\mathcal{O}(\mathcal{C})\cup\{y,y'\})}|S_{y,z}|^2.
\end{equation}
Therefore $S_{y,z}=0$ for all $z\not\cong y,y'\in\mathcal{O}(\mathcal{D})\setminus\mathcal{O}(\mathcal{C})$.  Moreover, up to permutation of simple objects, there exist signs $\varepsilon_1,\varepsilon_2,\varepsilon_3$ (determined by the $T$-matrix) such that the $S$-matrix of $\mathcal{D}$ is
\begin{equation}\label{eqs}
\left[\begin{array}{ccccccccccc}
1 & 1 & 1 & 1 & 2 & 2 & 2 & 2 & 2 & 2 & 2 \\
1 & 1 & 1 & 1 & 2 & 2 & 2 & -2 & -2 & -2 & -2\\
1 & 1 & 1 & 1 & 2 & -2 & -2 & 2 & 2 & -2 &-2 \\
1 & 1 & 1 & 1 & 2 & -2 & -2 & -2 & -2 & 2 & 2\\
2 & 2 & 2 & 2 & -4 & 0 & 0 & 0 & 0 & 0 & 0 \\
2 & 2 & -2 & -2 & 0 & \varepsilon_12\sqrt{2} & -\varepsilon_12\sqrt{2} & 0 & 0 & 0 & 0\\
2 & 2 & -2 & -2 & 0 & -\varepsilon_12\sqrt{2} & \varepsilon_12\sqrt{2} & 0 & 0 & 0 & 0\\
2 & -2 & 2 & -2 & 0 & 0 & 0 & \varepsilon_22\sqrt{2} & -\varepsilon_22\sqrt{2} & 0 & 0\\
2 & -2 & 2 & -2 & 0 & 0 & 0 & -\varepsilon_22\sqrt{2} & \varepsilon_22\sqrt{2} & 0 & 0\\
2 & -2 & -2 & 2 & 0 & 0 & 0 & 0 & 0 & \varepsilon_32\sqrt{2} & -\varepsilon_32\sqrt{2}\\
2 & -2 & -2 & 2 & 0 & 0 & 0 & 0 & 0 & -\varepsilon_32\sqrt{2} & \varepsilon_32\sqrt{2}
\end{array}\right].
\end{equation}

\par Let $\mathcal{P}$ be the unique (up to modular equivalence) pointed modular tensor category of rank 2 whose nontrivial simple object $x'$ has full twist $\theta_{x'}=\theta_x^{-1}$.  Then $\mathcal{D}\boxtimes\mathcal{P}$ contains a Lagrangian subcategory $\otimes$-generated by $\mathcal{D}_\mathrm{pt}$ and the simple object $x\boxtimes x'$.  This implies $\mathcal{D}\boxtimes\mathcal{P}\simeq\mathcal{Z}(\mathrm{Vec}_G^\omega)$ for a finite group $G$ of order 8 and a $3$-cocycle $\omega$ on $G$ \cite[Theorem 4.5]{drinfeld2007grouptheoretical}.  There are 38 braided equivalence classes of $\mathcal{Z}(\mathrm{Vec}_G^\omega)$ of this form \cite[Section 2.8]{MR3820602}.  But moreover, $\mathcal{D}\boxtimes\mathcal{P}$ must have conductor (Frobenius-Schur exponent) 16 and 12 simple objects whose twists are 16th roots of unity.  There are only 4 braided equivalence classes of $\mathcal{Z}(\mathrm{Vec}_G^\omega)$ with these characteristics which can be indexed as $\mathcal{Z}(\mathrm{Vec}_{Q_8}^\gamma)$ where $\gamma$ is any of the 4 generators of $H^3(Q_8,\mathbb{C}^\times)\cong\mathbb{Z}/8\mathbb{Z}$ \cite[Appendix A]{MR2333187}.  One can easily verify that each of $\mathcal{Z}(\mathrm{Vec}_{Q_8}^\gamma)$ factors as a Deligne product in this way, proving the following result.

\begin{lemma}\label{q8}
Let $\mathcal{C}:=\mathcal{C}(\chi_2^0,\tau)$ be a nonsymmetrically braided fusion category.  There exists a nondegenerate cover $\mathcal{D}$ of $\mathcal{C}$ with $\mathrm{FPdim}(\mathcal{D})=32$ if and only if there exists a generator $\gamma\in H^3(Q_8,\mathbb{C}^\times)$ and a braided equivalence $\mathcal{C}\simeq\mathcal{Z}(\mathrm{Vec}_{Q_8}^\gamma)_\mathrm{ad}$.
\end{lemma}

\par We have shown that if $\mathcal{D}$ is a nondegenerate cover of dimension $32$ of nonsymmetrically braided $\mathcal{C}(\chi_2^0,\tau)$, then $\mathcal{D}$ is a factor of one of the 4 doubles $\mathcal{Z}(\mathrm{Vec}_{Q_8}^\gamma)$ for a generator $\gamma\in H^3(Q_8,\mathbb{C}^\times)$.  Each $\mathcal{Z}(\mathrm{Vec}_{Q_8}^\gamma)$ factors like this in exactly 4 ways, depending on the choice of the order 2 invertible object to include in the other factor, which is necessarily rank 2 and pointed.  Therefore, there are at most 16 distinct minimal nondegenerate covers $\mathcal{D}$ over all $\mathcal{C}(\chi_2^0,\tau)$; in fact, there are exactly 16 as they are differentiated by their modular data.  We include the nontrivial $T$-eigenvalues in Figure \ref{fig:too} which determine the $S$-matrix in (\ref{eqs}).  As a result, we have recovered the original counterexample of the minimal modular extension conjecture due to V.\ Drinfeld (along with its reverse braiding).

\begin{proposition}
Let $\mathcal{C}$ be either of the $E_2$ Tambara-Yamagami braided fusion categories labeled $\mathcal{K}$ and $\mathcal{K}^\mathrm{rev}$ from Figure \ref{fig:onetwo}.  There does not exist a nondegenerate cover $\mathcal{D}$ of $\mathcal{C}$ with $\mathrm{FPdim}(\mathcal{D})=32$.
\end{proposition}

\begin{proof}
This follows from Lemma \ref{q8} along with \cite[Theorem 4.22]{MR3613518} and \cite[Theorem 1.1]{MR3613518}.  The latter two imply that a nonsymmetrically braided $E_2$ Tambara-Yamagami fusion category $\mathcal{C}(\chi_2^0,\tau)$ has exactly $|H^3(E_2,\mathbb{C}^\times)|=8$ nondegenerate extensions of dimension 32, up to equivalence.  Lemma \ref{q8} then implies only two such $\mathcal{C}(\chi_2^0,\tau)$ can possess nondegenerate covers of this dimension, which are distinguished by the twist on the noninvertible simple object $x$.
\end{proof}

\begin{figure}[H]
\centering
\begin{equation*}
\begin{array}{|c|cc|cc|cc|}
\hline \theta_x & \theta_{y_1} & \theta_{y_1'} & \theta_{y_2} & \theta_{y_2'} & \theta_{y_3'} & \theta_{y_3'} \\\hline
 i & \zeta_{16}^5 & -\zeta_{16}^5 & \zeta_{16}^5 & -\zeta_{16}^5 & \zeta_{16}^5 & -\zeta_{16}^5  \\
    & \zeta_{16} & -\zeta_{16} & \zeta_{16}^5 & -\zeta_{16}^5 & \zeta_{16} & -\zeta_{16}  \\
      & \zeta_{16}^5 & -\zeta_{16}^5 & \zeta_{16} & -\zeta_{16} & \zeta_{16} & -\zeta_{16}  \\
 & \zeta_{16} & -\zeta_{16} & \zeta_{16} & -\zeta_{16} & \zeta_{16}^5 & -\zeta_{16}^5  \\
\hline
  i & \zeta_{16}^5 & -\zeta_{16}^5 & \zeta_{16}^5 & -\zeta_{16}^5 & \zeta_{16} & -\zeta_{16}  \\
    & \zeta_{16} & -\zeta_{16} & \zeta_{16}^5 & -\zeta_{16}^5 & \zeta_{16}^5 & -\zeta_{16}^5  \\
     & \zeta_{16}^5 & -\zeta_{16}^5 & \zeta_{16} & -\zeta_{16} & \zeta_{16}^5 & -\zeta_{16}^5  \\
 & \zeta_{16} & -\zeta_{16} & \zeta_{16} & -\zeta_{16} & \zeta_{16} & -\zeta_{16}  \\
\hline
  -i & \zeta_{16}^7 & -\zeta_{16}^7 & \zeta_{16}^7 & -\zeta_{16}^7 & \zeta_{16}^7 & -\zeta_{16}^7  \\
     & \zeta_{16}^3 & -\zeta_{16}^3 & \zeta_{16}^7 & -\zeta_{16}^7 & \zeta_{16}^3 & -\zeta_{16}^3  \\
      & \zeta_{16}^7 & -\zeta_{16}^7 & \zeta_{16}^3 & -\zeta_{16}^3 & \zeta_{16}^3 & -\zeta_{16}^3  \\
  & \zeta_{16}^3 & -\zeta_{16}^3 & \zeta_{16}^3 & -\zeta_{16}^3 & \zeta_{16}^7 & -\zeta_{16}^7  \\
\hline
   -i & \zeta_{16}^7 & -\zeta_{16}^7 & \zeta_{16}^7 & -\zeta_{16}^7 & \zeta_{16}^3 & -\zeta_{16}^3  \\
     & \zeta_{16}^3 & -\zeta_{16}^3 & \zeta_{16}^7 & -\zeta_{16}^7 & \zeta_{16}^7 & -\zeta_{16}^7  \\
       & \zeta_{16}^7 & -\zeta_{16}^7 & \zeta_{16}^3 & -\zeta_{16}^3 & \zeta_{16}^7 & -\zeta_{16}^7  \\
 & \zeta_{16}^3 & -\zeta_{16}^3 & \zeta_{16}^3 & -\zeta_{16}^3 & \zeta_{16}^3 & -\zeta_{16}^3  \\
 \hline
\end{array}
\end{equation*}
    \caption{$T$-matrices of minimal nondegenerate covers of $\mathcal{Z}(\mathrm{Vec}^\gamma_{Q_8})_\mathrm{ad}$ for generators $\gamma\in H^3(Q_8,\mathbb{C}^\times)$}%
    \label{fig:too}%
\end{figure}


\subsection{$\chi_2^1$ braidings}

Here we consider the 12 braided equivalence classes of $\mathcal{C}(\chi_2^1,\tau)$.  Recall that the pointed fusion subcategory corresponding to $E_2$ is symmetric in any case, and $q(g_j)^2=-1$ for $j=1,2$, hence $g_1,g_2$ do not lie in the symmetric center of $\mathcal{C}(\chi_2^1,\tau)$, while $g_1+g_2$ does as $q(g_1+g_2)^2=1$.  We have $\chi^1_2(g_1+g_2,g_1+g_2)=1$, hence the symmetric center of $\mathcal{C}(\chi_2^1,\tau)$ is braided equivalent to $\mathrm{Rep}(\mathbb{Z}/2\mathbb{Z},e)$ for any $\chi_2^1$ braiding.

\par Let $\mathcal{I},\mathcal{I}'$ be any of the 8 Ising braided fusion categories and denote their isomorphism classes of simple objects by $\{e,g,x\}$ and $\{e',g',x'\}$ where $x,x'$ are the unique noninvertible isomorphism classes.  These categories were described in Example \ref{ice}.  In particular each Ising braided fusion category is distinguished by a primitive 16th root of unity $\zeta$.  It is clear that $\mathcal{I}\boxtimes\mathcal{I}'$ is self-dual and $\mathcal{I}\boxtimes\mathcal{I}'$ contains a maximal integral subcategory $(\mathcal{I}\boxtimes\mathcal{I}')_\mathbb{Q}$ with four invertible simple objects $(\mathcal{I}\boxtimes\mathcal{I}')_\mathrm{pt}$ along with the simple object $x\boxtimes x'$ of Frobenius-Perron dimension $2$.  The subcategory $(\mathcal{I}\boxtimes\mathcal{I}')_\mathbb{Q}$ can also be identified as the relative centralizer of the Tannakian subcategory generated by $g\boxtimes g'$.  Therefore $(\mathcal{I}\boxtimes\mathcal{I}')_\mathbb{Q}$ is a braided fusion category of the form $\mathcal{C}(\chi_2^1,\tau)$.

\par Conversely, assume a braided fusion category $\mathcal{C}:=\mathcal{C}(\chi_2^1,\tau)$ is given with isomorphism classes of simple objects $\{e,g_1,g_2,g_1+g_2,x\}$.  Without loss of generality, consider $\mathcal{C}$ as a modular tensor category with its unique positive spherical structure.   Assume $\mathcal{D}$ is a nondegenerate cover of $\mathcal{C}$ with $\dim(\mathcal{D})=16$ so that $\mathcal{C}=C_\mathcal{D}(C_\mathcal{C}(\mathcal{C}))$ by Lemma \ref{previous}.  If $\dim(\mathcal{D}_\mathrm{pt})>\dim(\mathcal{C}_\mathrm{pt})=4$ then the adjoint subcategory of $\mathcal{D}$ would have dimension $2$ or $1$.  But $\mathcal{C}_\mathrm{pt}\subset\mathcal{D}_\mathrm{ad}$, so we may conclude that each universally graded component has dimension 4 with $\mathcal{C}_\mathrm{pt}=\mathcal{D}_\mathrm{ad}$ as the trivially graded component.  As $\mathcal{D}$ is nilpotent and weakly integral, $\dim(y)^2\in\{2,4\}$ for all noninvertible $y\in\mathcal{O}(\mathcal{D})$.  Any object of dimension 2 is unique in its universally graded component so it lies in $C_\mathcal{D}(C_\mathcal{C}(\mathcal{C}))=\mathcal{C}$ by the balancing equation \cite[Proposition 8.13.8]{tcat} since $\theta_{g_1+g_2}=1$.  Moreover $x\in\mathcal{O}(\mathcal{D})$ is the unique simple object of dimension 2 while the other 2 nontrivial components have 2 isomorphism classes of simple objects of dimension $\sqrt{2}$.  Let $y\not\cong y'$ be simple objects of dimension $\sqrt{2}$ in a nontrivially graded component of $\mathcal{D}$.  If $y^\ast\cong y'$, we still have $\theta_y=\theta_{y'}$.  Hence $S_{g_1+g_2,y}=S_{g_1+g_2,y'}=\sqrt{2}$ by the balancing equation which again implies $y,y'\in C_\mathcal{D}(C_\mathcal{C}(\mathcal{C}))=\mathcal{C}$, a contradiction.  Therefore we conclude that $\mathcal{D}$ is self-dual.  Moreover, any simple object of dimension $\sqrt{2}$ $\otimes$-generates an Ising braided fusion category which is necessarily nondegenerately braided.  Hence $\mathcal{D}$ factors as a product of Ising braided fusion categories \cite[Theorem 3.13]{DGNO}, finishing the proof of the following fact.

\begin{lemma}
Let $\mathcal{C}:=\mathcal{C}(\chi_2^1,\tau)$ be given.  There exists a nondegenerate cover $\mathcal{D}$ of $\mathcal{C}$ with $\mathrm{FPdim}(\mathcal{D})=16$ if and only if there exist Ising braided fusion categories $\mathcal{I},\mathcal{I}'$ and a braided equivalence $\mathcal{C}\simeq(\mathcal{I}\boxtimes\mathcal{I}')_\mathbb{Q}$.
\end{lemma}

\par There are at most $\binom{8+2-1}{2}=36$ distinct products $\mathcal{I}\boxtimes\mathcal{I}'$ up to braided equivalence since $\boxtimes$ is symmetric, which we can sort by Witt equivalence \cite[Definition 5.1]{DMNO} prior to sorting each Witt equivalence class by braided equivalence.  If $\mathcal{I}=\mathcal{I}_j$ and $\mathcal{I}'=\mathcal{I}_k$ for some primitive 16th roots of unity $\zeta=\zeta_{16}^j,\zeta'=\zeta_{16}^k$, then the Witt equivalence classes of products $\mathcal{I}\boxtimes\mathcal{I}'$ are indexed by 8th roots of unity $\xi$ via $\xi:=f(\zeta)f(\zeta')/2$ where $f(\zeta):=\zeta^{-1}(\zeta^2+\zeta^{-2})$ \cite[Lemma B.24]{DGNO}.  When $\mathcal{I},\mathcal{I}'$ are equipped with their unique positive spherical structure, this is the multiplicative central charge \cite[Equation (8.60)]{tcat} of $\mathcal{I}\boxtimes\mathcal{I}'$.  Among Witt equivalence classes many of these products are braided equivalent.  In particular it is easy to check that in $\mathcal{I}\boxtimes\mathcal{I}'$, any of the 4 simple objects of dimension $\sqrt{2}$ $\otimes$-generates an Ising braided fusion category $\mathcal{I}''$ and $\mathcal{I}\boxtimes\mathcal{I}'\simeq \mathcal{I}''\boxtimes C_{\mathcal{I}\boxtimes\mathcal{I}'}(\mathcal{I}'')$ is a braided equivalence which implies there exist at most 2 distinct nontrivial factorizations of this type.  In particular, $\mathcal{I}_j\boxtimes\mathcal{I}_k\simeq\mathcal{I}_{j+8}\boxtimes\mathcal{I}_{k+8}$ is a braided equivalence where $j+8,k+8$ are considered modulo 16.  Once accounting for this symmetry, the remaining 20 braided equivalence classes of categories $\mathcal{I}\boxtimes\mathcal{I}'$, which we collect in column 3 of Figure \ref{fig:six}, are distinguished by the modular data associated to their unique positive spherical structure.

\par We further distinguish the braided equivalence classes of $\mathcal{I}\boxtimes\mathcal{I}'$ in column 4 of Figure \ref{fig:six} by identifying those whose integral braided fusion subcategories are equivalent.  This is straightforward because as fusion categories, $(\mathcal{C}(\chi_1^1,\tau)\boxtimes\mathcal{C}(\chi_1^1,\tau'))_\mathbb{Q}\simeq\mathcal{C}(\chi_2^1,\tau\tau')$ (see Lemma \ref{explicit}).  This shows there are at least 12 braided equivalence classes of $(\mathcal{I}\boxtimes\mathcal{I}')_\mathbb{Q}$.  As there are 12 braided equivalence classes of braided fusion categories $\mathcal{C}(\chi_2^1,\tau)$, then there are exactly 12 braided equivalence classes of $(\mathcal{I}_j\boxtimes\mathcal{I}_k)_\mathbb{Q}$ across all products $\mathcal{I}\boxtimes\mathcal{I}'$.

\begin{figure}[H]
\centering
\begin{equation*}
\begin{array}{|c|c|c|c|}\hline
\xi & (j,k) & \mathcal{I}_j\boxtimes\mathcal{I}_k\text{ br.eq. \!classes} & (\mathcal{I}_j\boxtimes\mathcal{I}_k)_\mathbb{Q}\text{ br.eq. \!classes}\\\hline
1 & (1,15),(3,13),(5,11),(7,9) & (1,15),(3,13) & (1,15) \\
\zeta_8 &  (1,5),(3,11),(7,7),(9,13),(15,15) & (1,5),(3,11),(7,7) & (1,5),(7,7) \\
i & (1,3),(5,15),(7,13),(9,11) & (1,3),(5,15) & (1,3)\\
\zeta_8^3 & (1,9),(3,15),(5,5),(7,11),(13,13) & (1,9),(3,15),(5,5) & (1,9),(5,5)\\
-1 & (1,7),(3,5),(9,15),(11,13) & (1,7),(3,5) & (1,7)\\
\zeta_8^5 & (1,13),(3,3),(5,9),(7,15),(11,11) & (1,13),(3,3),(7,15) & (1,13),(3,3) \\
-i & (1,11),(3,9),(5,7),(13,15) & (1,11),(5,7) & (1,11)\\
\zeta_8^7 & (1,1),(3,7),(5,13),(9,9),(11,15) & (1,1),(3,7),(5,13) & (1,1),(3,7) \\
\hline
\end{array}
\end{equation*}
    \caption{Braided equivalence classes of $\mathcal{I}\boxtimes\mathcal{I}'$ and $(\mathcal{I}\boxtimes\mathcal{I}')_\mathbb{Q}$}%
    \label{fig:six}%
\end{figure}

\begin{example}\label{structure}
The take-away from Figure \ref{fig:six} is that of the 12 braided equivalence classes of categories $\mathcal{C}(\chi_2^1,\tau)$, 8 have 2 minimal nondegenerate covers up to equivalence, and 4 have a unique minimal nondegenerate cover up to equivalence.  This is seemingly at odds with \cite[Theorem 1.1]{MR3613518} combined with \cite[Theorem 4.22]{MR3613518} which, since $\mathcal{C}(\chi_2^1,\tau)$ have unitary structures and symmetric center $\mathrm{Rep}(\mathbb{Z}/2\mathbb{Z})$, imply that there are exactly 2 minimal nondegenerate extensions of $\mathcal{C}(\chi_2^1,\tau)$ up to equivalence.  The difference between the minimal nondegenerate covers of $\mathcal{C}(\chi_2^1,\tau)$ and the minimal nondegenerate extensions of $\mathcal{C}(\chi_2^1,\tau)$ lies in the braided autoequivalences of $\mathcal{C}(\chi_2^1,\tau)$.

\par Recall \cite[Proposition 1]{MR1776075} that there is exactly one nontrivial element $P\in\mathrm{Aut}_\otimes^\mathrm{br}(\mathcal{C}(\chi_2^1,\tau))\cong\mathrm{Aut}_\otimes^\mathrm{br}((\mathcal{I}\boxtimes\mathcal{I}')_\mathbb{Q})$ which is the strict autoequivalence permuting $g_1\leftrightarrow g_2$.  Now assume that $P$ lifts to a braided autoequivalence $\tilde{P}$ of the cover $\mathcal{I}\boxtimes\mathcal{I}'$.  Then the obvious extension $(\mathcal{I}\boxtimes\mathcal{I}',\iota)$ given by the chosen basis $g_1,g_2$ of $E_2$ is equivalent to $(\mathcal{I}\boxtimes\mathcal{I}',\iota\circ P)$ precisely via the braided autoequivalence $\tilde{P}$.  Conversely, assume $F:(\mathcal{I}\boxtimes\mathcal{I}',\iota)\to(\mathcal{I}\boxtimes\mathcal{I}',\kappa)$ is an equivalence of extensions.  Then $\left.F\right|_{(\mathcal{I}\boxtimes\mathcal{I}')_\mathbb{Q}}$
is a braided autoequivalence of $(\mathcal{I}\boxtimes\mathcal{I}')_\mathbb{Q}$.  In particular, $P$ acts trivially on equivalence classes of minimal nondegenerate extensions of $(\mathcal{I}\boxtimes\mathcal{I}')_\mathbb{Q}$ if and only if $P$ lifts to a braided equivalence $\tilde{P}$ of $\mathcal{I}\boxtimes\mathcal{I}'$.  Any lifting $\tilde{P}$ of the braided autoequivalence permuting $g_1\leftrightarrow g_2$ must nontrivally permute the simple objects of dimension $\sqrt{2}$ (so the fusion rules coincide).  In the cases where $\mathcal{I}\boxtimes\mathcal{I}'$ has a unique minimal nondegenerate cover, all simple objects of dimension $\sqrt{2}$ have distinct twists, hence $(\mathcal{I}\boxtimes\mathcal{I}',\iota)$ and $(\mathcal{I}\boxtimes\mathcal{I}',P\circ\iota)$ are inequivalent as extensions.  In the cases where $\mathcal{I}\boxtimes\mathcal{I}'$ has two inequivalent minimal nondegenerate covers, $\zeta=\zeta'$ or $\zeta=-\zeta'$ which determines the lifting $\tilde{P}$ on the level of objects.  We graphically represent the two distinct situations in Figure \ref{fig:seven} for $\xi=\zeta_8$.
\end{example}

\begin{figure}[H]
\centering
\begin{tikzpicture}
\node at (-3,0) (a) {$(\mathcal{I}_1\boxtimes\mathcal{I}_5)_\mathbb{Q}$};
\node at (-3,3) (b) {$\mathcal{I}_1\boxtimes\mathcal{I}_5$};
\draw[->] (a) edge [in=330, out=30] (b);
\draw[->] (a) edge [in=210, out=150] (b);
\draw[<->] (-4,1.5) -- node[above] {$P$} (-2,1.5);
\node at (3,0) (c) {$(\mathcal{I}_7\boxtimes\mathcal{I}_7)_\mathbb{Q}$};
\node at (1,3) (d) {$\mathcal{I}_3\boxtimes\mathcal{I}_{11}$};
\node at (5,3) (e) {$\mathcal{I}_7\boxtimes\mathcal{I}_7$};
\draw[->] (c) -- (d);
\draw[->] (c) -- (e);
\node at (2,1.5) (f) {};
\node at (4,1.5) (g) {};
\draw[->] (f) edge [in=180, out=270,looseness=20] node[below left] {$P$} (f);
\draw[->] (g) edge [in=0, out=270,looseness=20] node[below right] {$P$} (g);
\end{tikzpicture}
    \caption{Minimal nondegenerate extensions of $(\mathcal{I}\boxtimes\mathcal{I}')_\mathbb{Q}$ with $\xi=\zeta_8$}%
    \label{fig:seven}%
\end{figure}
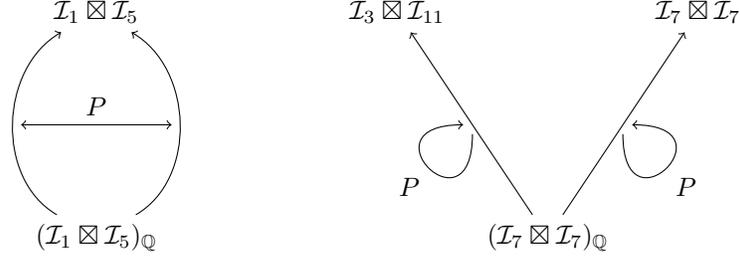

\section{General Tambara-Yamagami braided fusion categories}\label{jen}

If $\mathcal{C}(\chi_n^k,\tau)$ is symmetrically braided, then by the formulas for the braidings in Section \ref{TY}, we must have $k=0$.  All 4 of these symmetrically braided fusion categories possess minimal nondegenerate covers with $\mathrm{FPdim}(\mathcal{C}(\chi_n^k,\tau))^2=2^{4n+2}$, e.g.\ their doubles, $\mathcal{Z}(\mathcal{C}(\chi_n^k,\tau))$.  The remainder of this section describes existence/nonexistence of minimal nondegenerate covers in the nonsymmetrically braided case.  We prove in Section \ref{non} that there do not exist nondegenerate covers of $\mathcal{C}(\chi_{2n}^0,\tau)$ with Frobenius-Perron dimension $2^{4n+1}$.  We prove in Section \ref{nonnon} that there exist minimal nondegenerate covers of all $\mathcal{C}(\chi_n^1,\tau)$. 


\subsection{Nonsymmetric $\chi_{2n}^0$ braidings}\label{non}

The proof of the following lemma is a generalization of the explanation in Section \ref{firstexample}.

\begin{lemma}\label{begin}
Let nonsymmetrically braided $\mathcal{C}:=\mathcal{C}(\chi_{2n}^0,\tau)$ be given for some $n\in\mathbb{Z}_{\geq1}$.  Let $\mathcal{D}$ be a nondegenerate cover of $\mathcal{C}$ with $\mathrm{FPdim}(\mathcal{D})=2^{4n+1}$ and consider $\mathcal{D}$ as a modular tensor category equipped with its unique positive spherical structure.  Then $\theta_y$ is a primitive 16th root of unity for all $y\in\mathcal{O}(\mathcal{D})\setminus\mathcal{O}(\mathcal{C})$.
\end{lemma}

\begin{proof}
Let $\mathcal{C}:=\mathcal{C}(\chi_{2n}^0,\tau)$ be given with $\mathcal{O}(\mathcal{C}_\mathrm{pt})=E_{2n}=(\mathbb{Z}/2\mathbb{Z})^{\oplus 2n}$ and a unique isomorphism class of simple objects $x$.  We have $C_\mathcal{C}(\mathcal{C})=\mathcal{C}_\mathrm{pt}$ is Tannakian with $\mathrm{FPdim}(\mathcal{C}_\mathrm{pt})=2^{2n}$; denote the regular algebra by $R$.  Let $\mathcal{D}$ be a nondegenerate cover with $\mathrm{FPdim}(\mathcal{D})=2^{4n+1}$.  When $\mathcal{D}$ (and thus $\mathcal{C}$) are equipped with their unique positive spherical structure, then $\mathcal{D}_R^0$ is a rank 2 pointed modular tensor category by \cite[Corollary 3.32]{DMNO}, hence its multiplicative central charge is $\xi(\mathcal{D})=\zeta_8^{\pm}$ \cite[Remark 5.29]{DMNO}.  On the other hand, $C_\mathcal{D}(\mathcal{C}_\mathrm{pt})=\mathcal{C}$ by Lemma \ref{previous}, hence $\xi(\mathcal{D})\sqrt{2}=1+\theta_x$ from \cite[Example 3.14]{DMNO} and the definition \cite[Section 6.2]{DGNO}.  We reserve this fact for later use.
\par Also, since $C_\mathcal{D}(\mathcal{D}_\mathrm{pt})=\mathcal{D}_\mathrm{ad}$ \cite[Corollary 8.22.8]{tcat}, if $\dim(\mathcal{D}_\mathrm{pt})\geq 2^{2n+1}$, then $\dim(\mathcal{D}_\mathrm{ad})\leq 2^{2n}$ by Equation (\ref{wanpisu}).  But $\mathcal{C}_\mathrm{pt}\subset\mathcal{D}_\mathrm{ad}$, therefore $\mathcal{D}_\mathrm{ad}=\mathcal{C}_\mathrm{pt}$ in this case.  Therefore,
\begin{equation}
\mathcal{C}=C_\mathcal{C}(\mathcal{C}_\mathrm{pt})\subset C_\mathcal{D}(\mathcal{C}_\mathrm{pt})=C_\mathcal{D}(\mathcal{D}_\mathrm{ad})=\mathcal{D}_\mathrm{pt},
\end{equation}
a contradiction.  Moreover we may conclude that $\mathcal{D}_\mathrm{pt}=\mathcal{C}_\mathrm{pt}$ and thus $\mathcal{D}_\mathrm{ad}=\mathcal{C}$.

\par Now let $g\in E_{2n}$ and $y\in\mathcal{O}(\mathcal{D})$. The balancing equation states that $S_{g,y}=\theta_g^{-1}\theta_y^{-1}\theta_{g\otimes y}\dim(g\otimes y)$, hence $|S_{g,y}|^2=\dim(y)^2$.  Also, $S_{x,x}=S_{x,x^\ast}=\theta_x^{-2}2^{2n}$.  Thus $|S_{x,x}|^2=2^{4n}$.  The orthogonality relation of $x$ with itself yields
\begin{align}
\dim(\mathcal{D})&=\sum_{g\in E_{2n}}|S_{g,x}|^2+|S_{x,x}|^2+\sum_{y\in\mathcal{O}(\mathcal{D})\setminus\mathcal{O}(\mathcal{C})}|S_{y,x}|^2 \\
&=2^{2n}\cdot 2^{2n}+2^{4n}+\sum_{y\in\mathcal{O}(\mathcal{D})\setminus\mathcal{O}(\mathcal{C})}|S_{y,x}|^2 \\
&=\dim(\mathcal{D})+\sum_{y\in\mathcal{O}(\mathcal{D})\setminus\mathcal{O}(\mathcal{C})}|S_{y,x}|^2.
\end{align}
Therefore $\sum_{y\in\mathcal{O}(\mathcal{D})\setminus\mathcal{O}(\mathcal{C})}|S_{y,x}|^2=0$ and moreover $S_{x,y}=S_{y,x}=0$ for every $y\in\mathcal{O}(\mathcal{D})\setminus\mathcal{O}(\mathcal{C})$.  Verlinde formula then implies for all $y\in\mathcal{O}(\mathcal{D})\setminus\mathcal{O}(\mathcal{C})$,
\begin{equation}
\dim(\mathcal{D})N_{x,y}^y=\sum_{g\in E_{2n}}\dfrac{S_{x,g}S_{y,g}\overline{S_{y,g}}}{\dim(g)}=\sum_{g\in E_{2n}}S_{x,g}|S_{y,g}|^2,
\end{equation}
since for every $z\in\mathcal{O}(\mathcal{D})\setminus E_{2n}$, we have shown either $S_{x,z}=0$ or $S_{y,z}=0$.  Moreover, we have already computed the remaning $S$-matrix entries above, implying
\begin{equation}
N^x_{y,y^\ast}=N_{x,y}^y=\dfrac{1}{2^{4n+1}}\sum_{g\in E_{2n}}S_{x,g}|S_{y,g}|^2
=\dfrac{1}{2^{4n+1}}\sum_{g\in E_{2n}}2^n\dim(y)^2=\dfrac{\dim(y)^2}{2^{n+1}}.
\end{equation}
Since $y\otimes y^\ast\cong N_{y,y^\ast}^xx\oplus\bigoplus_{g\in E_{2n}}N_{y,y^\ast}^gg$, we may compute the dimension of $y\otimes y^\ast$ as
\begin{align}
&&\dim(y)^2&=\sum_{g\in E_{2n}}N_{y,y^\ast}^g+\dfrac{\dim(y)^2}{2^{n+1}}2^n \\
\Rightarrow&&\sum_{g\in E_{2n}}N_{y,y^\ast}^g&=\dim(y)^2/2,
\end{align}
which is the order of the stabilizer subgroup of $y$ under the $\otimes$-action of $E_{2n}$.  The orbit-stabilizer thereom then implies that the cardinality of the orbit of $y$ is $2^{2n+1}/\dim(y)^2$.  Moreover, by computing the dimension of this orbit, by \cite[Theorem 3.5.2]{tcat} each nontrivial graded component of $\mathcal{D}$ consists of the orbit of a single simple object under the $\otimes$-action of $E_{2n}$.  In particular, this implies there exists $g\in E_{2n}$ such that $ y^\ast\cong g\otimes y$ since the universal grading group is an elementary abelian $2$-group.
\par Lastly we compute with the balancing equation \cite[Proposition 8.13.8]{tcat},
\begin{align}
S_{y,y^\ast}&=\theta_y^{-2}\left(\sum_{g\in E_{2n}}N_{y,y^\ast}^g\dim(g)\theta_g+N_{y,y^\ast}^x\dim(x)\theta_x\right) \\
&=\theta_y^{-2}\left(\dfrac{\dim(y)^2}{2}+\dfrac{\dim(y)^2}{2^{n+1}}2^n\theta_x\right) \\
&=\dfrac{\dim(y)^2\xi(\mathcal{D})\sqrt{2}}{2\theta_y^2}.\label{twentyfive}
\end{align}
As $y^\ast\cong g\otimes y$ for $g\in G$, then by a trivial application of \cite[Proposition 8.13.10]{tcat} and recalling that $\theta_{y^\ast}=\theta_y$, we have
\begin{equation}
\overline{S_{y,y}}=S_{y,y^\ast}=S_{y,g\otimes y}=\dfrac{1}{\dim(y)}S_{g,y}S_{y,y}=\dfrac{\theta_{g\otimes y}}{\theta_y}S_{y,y}=\dfrac{\theta_{y^\ast}}{\theta_y}S_{y,y}=S_{y,y}.
\end{equation}
Moreover $S_{y,y}=S_{y,y^\ast}$ is real and we may conclude from Equation (\ref{twentyfive}) that $\theta_y^2=\pm\xi(\mathcal{D})$, i.e.\ $\theta_y$ is a primitive $16$th root of unity for all $y\in\mathcal{O}(\mathcal{D})\setminus\mathcal{O}(\mathcal{C})$.
\end{proof}

\begin{lemma}\label{int}
Let nonsymmetrically braided $\mathcal{C}:=\mathcal{C}(\chi_{2n}^0,\tau)$ be given for some $n\in\mathbb{Z}_{\geq1}$.  If $\mathcal{D}$ is a nondegenerate cover of $\mathcal{C}$ with $\mathrm{FPdim}(\mathcal{D})=2^{4n+1}$, then $\mathcal{D}$ is integral.
\end{lemma}

\begin{proof}
As $\mathcal{D}$ is a weakly integral modular tensor category, it may be equipped with its unique positive spherical structure, so that for all $y\in\mathcal{O}(\mathcal{D})$, $\mathbb{Q}(\mathrm{FPdim}(y))$ is either $\mathbb{Q}$ or $\mathbb{Q}(\sqrt{2})$ by \cite[Proposition 8.14.6]{tcat} and \cite[Proposition 1.4]{2019arXiv191212260G}.  If there exists $z\in\mathcal{O}(\mathcal{D})$ with $\mathbb{Q}(\mathrm{FPdim}(z))=\mathbb{Q}(\sqrt{2})$, then the graded component containing $z$ along with the trivial component $\mathcal{D}_\mathrm{ad}=\mathcal{C}$, generate a fusion subcategory $\mathcal{E}\subset\mathcal{D}$ with $\mathrm{FPdim}(\mathcal{E})=2^{2n+2}$, as the universal grading group has exponent 2.  From the fact that $\mathcal{C}\subsetneq\mathcal{E}$,  $C_\mathcal{D}(\mathcal{E})$ is pointed with $\mathrm{FPdim}(C_\mathcal{D}(\mathcal{E}))=2^{2n-1}$; let $R$ be the regular algebra.  Then $\mathrm{FPdim}(\mathcal{D}_R^0)=8$ by \cite[Corollary 3.32]{DMNO} and the Frobenius-Perron dimensions of simple objects are known because $\mathcal{O}(\mathcal{D}_R^0)$ corresponds to the isomorphism classes of simple summands of the free $R$-modules $R\otimes w$ for $w\in\mathcal{O}(\mathcal{E})$ (Section \ref{doubles}).  There are two nonisomorphic invertible objects in $\mathcal{O}(\mathcal{D}_R^0)$ corresponding to the free $R$-modules on $w\in\mathcal{O}(\mathcal{D}_\mathrm{pt})$.  The simple summands of $R\otimes x$ must all be isomorphic with Frobenius-Perron dimension 2, or else the sum of their dimensions is greater than or equal to 8.  The remaining simple objects of $\mathcal{D}_R^0$ correspond to the (necessarily isomorphic) simple summands of the free $R$-module $R\otimes z$ which must have Frobenius-Perron dimension $\sqrt{2}$.  Moreover the list of dimensions of $\mathcal{D}_R^0$ is $1,1,2,\sqrt{2}$.  But no such fusion category exists since $\mathrm{FPdim}((\mathcal{D}_R^0)_\mathbb{Q})=6$ does not divide $\mathrm{FPdim}(\mathcal{D}_R^0)=8$, violating \cite[Proposition 8.15]{ENO}.

\end{proof}

\begin{proposition}\label{bigprop}
Let nonsymmetrically braided $\mathcal{C}:=\mathcal{C}(\chi_{2n}^0,\tau)$ be given for some $n\in\mathbb{Z}_{\geq1}$.  If there exists a nondegenerate cover $\mathcal{D}$ of $\mathcal{C}$ with $\mathrm{FPdim}(\mathcal{D})=2^{4n+1}$, then $n=1$.
\end{proposition}

\begin{proof}
Let $\mathcal{C}:=\mathcal{C}(\chi_{2n}^0,\tau)$ be given and assume $\mathcal{D}$ is a nondegenerate cover of $\mathcal{C}$ with $\mathrm{FPdim}(\mathcal{D})=2^{4n+1}$, which is a modular tensor category equipped with its unique positive spherical structure.  Set $\mathcal{P}$ to be the pointed modular tensor category of rank 2 such that $\xi(\mathcal{D})=\xi(\mathcal{P})^{-1}$.  Note that this implies the nontrivial simple object $h\in\mathcal{O}(\mathcal{P})$ has twist $\theta_h=\theta^{-1}_x$ by construction, where $x\in\mathcal{O}(\mathcal{C})$ is the unique noninvertible simple object.   Lemma \ref{int} ensures that $\mathcal{D}$ is integral, hence $\mathcal{D}\boxtimes\mathcal{P}\simeq\mathcal{Z}(\mathrm{Vec}_G^\omega)$ is a braided equivalence for some finite group $G$ of order $2^{2n+1}$ and $\omega\in H^3(G,\mathbb{C}^\times)$ \cite[Theorem 1.3]{drinfeld2007grouptheoretical}.  Lemma \ref{begin} implies that $x\boxtimes h$ is the unique noninvertible simple object $z\in\mathcal{O}(\mathcal{D}\boxtimes\mathcal{P})$ such that $\theta_z=1$, while $g\boxtimes e$ for $g\in E_{2n}$ are the only invertible objects with trivial twist.  Therefore the fusion subcategory generated by these simple objects is the unique Lagrangian subcategory of $\mathcal{Z}(\mathrm{Vec}_G^\omega)$, and has the fusion rules of the character ring of an extraspecial 2-group.  Moreover $G$ is isomorphic to an extraspecial $2$-group by Lemma \ref{extraspecial}.  If $n>1$ this is impossible due to Lemma \ref{begin} since the conductor (Frobenius-Schur exponent) of $\mathcal{Z}(\mathrm{Vec}_G^\omega)$ is less than or equal to 8 for extra-special $2$-groups $G$ of order greater than $2^3$ \cite[Theorem 4.7]{MR2333187}.
\end{proof}


\subsection{$\chi_n^1$ braidings}\label{nonnon}

Let $\mathcal{C}:=\mathcal{C}(\chi_n^1,\tau)$ be given.  Recall from Section \ref{TY} that $\mathrm{FPdim}(\mathcal{C})=2^{n+1}$ and $\mathrm{FPdim}(C_\mathcal{C}(\mathcal{C}))=2^{n-1}$.  Here we will describe a nondegenerate cover of $\mathcal{C}(\chi_n^1,\tau)$ with Frobenius-Perron dimension $2^{2n}$ which is the absolute minimum possible acccording to Equation (\ref{wanpisu}).

\begin{lemma}\label{explicit}
Let $n\in\mathbb{Z}_{\geq1}$ and $\tau_j=\pm1/\sqrt{2}$ for $1\leq j\leq n$.  Let $\mathcal{C}(\chi_1^1,\tau_j)$ be Ising braided fusion categories with braiding data $q_j:E_1\times E_1\to\mathbb{C}^\times$ and $\alpha_j\in\mathbb{C}$ as in Section \ref{TY}.  There is an equivalence of braided fusion categories
\begin{equation}\label{right}
\left(\boxtimes_{j=1}^n\mathcal{C}(\chi_1^1,\tau_j)\right)_0\simeq\mathcal{C}\left(\chi^1_n,\prod_{j=1}^n\tau_j\right)
\end{equation}
where $\left(\boxtimes_{j=1}^n\mathcal{C}(\chi_1^1,\tau_j)\right)_0$ is the braided fusion subcategory generated by the unique simple object of maximal dimension.  The braiding data for the righthand side of the equivalence in (\ref{right}) is given by $q:E_n\times E_n\to\mathbb{C}^\times$ defined on the basis $g_j\in\mathcal{O}(\mathcal{C}(\chi_1^1,\tau_j)_\mathrm{pt})$ as $q(g_j):=q_j(g_j)$, and $\alpha:=\prod_{j=1}^n\alpha_j$.
\end{lemma}

\begin{proof}
Set $\mathcal{D}:=\boxtimes_{j=1}^n\mathcal{C}(\chi_1^1,\tau_j)$ and let $x_j\in\mathcal{O}(\mathcal{C}(\chi_1^1,\tau_j))$ be the noninvertible simple object for all $1\leq j\leq n$.  As Frobenius-Perron dimension is multiplicative across $\boxtimes$, there exists a unique simple object $x:=x_1\boxtimes\cdots\boxtimes x_n\in\mathcal{O}(\mathcal{D})$ of squared Frobenius-Perron dimension $2^n$.  Therefore $g\otimes x\cong x$ for all $g\in\mathcal{O}(\mathcal{D}_\mathrm{pt})$.   Hence $x\otimes x\cong \oplus_{g\in\mathcal{O}(\mathcal{D}_\mathrm{pt})}g$, and $x$ $\otimes$-generates a braided fusion subcategory $\mathcal{D}_0$ braided equivalent to an $E_n$ Tambara-Yamagami fusion category $\mathcal{C}(\chi,\tau)$ corresponding to some $\chi:E_n\times E_n\to\mathbb{C}^\times$ and $\tau=\pm1/2^{n/2}$.  Since $\mathcal{D}_0$ is braided, $\chi(g_j,g_j)=-1$ for $1\leq j\leq n$ as this is the braiding of $g_j$ with itself in $\mathcal{C}(\chi_1^1,\tau_j)$.  By the classification of symmetric nondegenerate bilinear forms on $E_n$ \cite[Section 5]{MR156890}, we must have $\chi=\chi_n^1$.  Now recall that by definition the square of the braiding of $x_j$ with itself is multiplication by $\alpha^2_j=\tau_j(1+q_j(g_j))$ for all $1\leq j\leq n$, and the square of the braiding of $x$ with itself is multiplication by $\alpha^2=\tau\sum_{g\in\mathcal{O}(\mathcal{D}_\mathrm{pt})}q(g)$.  Since $\mathcal{C}(\chi_1^1,\tau_j)$ for $1\leq j\leq n$ centralize one another pairwise, then $\alpha=\alpha_1\cdots\alpha_n$.  This implies
\begin{equation}
\tau\sum_{g\in\mathcal{O}(\mathcal{D}_\mathrm{pt})}q(g)=\alpha^2=\prod_{j=1}^n\alpha_j^2=\prod_{j=1}^n\tau_j\left(1+q_j(g_j)\right)=\prod_{j=1}^n\tau_j\sum_{g\in \mathcal{O}(\mathcal{D}_\mathrm{pt})}q(g),
\end{equation}
when $q$ is defined as in the statement of the lemma.  Moreover $\tau=\prod_{j=1}^n\tau_j$.
\end{proof}

\begin{proposition}
Let $\mathcal{C}:=\mathcal{C}(\chi_n^1,\tau)$ be given for $n\in\mathbb{Z}_{\geq1}$.  Then there exists an $n$-fold Deligne product of Ising braided fusion categories which is a minimal nondegenerate cover of $\mathcal{C}$.
\end{proposition}

\begin{proof}
For $n=1$ the statement is trivial so let $n\geq2$.  Consider any $n$-fold Deligne product $\mathcal{D}:=\boxtimes_{j=1}^n\mathcal{C}(\chi_1^1,\tau_j)$.  We need only show that there exists some choice of $\tau_j$ and braiding data $q_j$ and $\alpha_j$ for $1\leq j\leq n$ such that $\mathcal{C}$ and $\mathcal{D}_0$ are braided equivalent where $\mathcal{D}_0$ is defined in Lemma \ref{explicit}.  To this end, define $\tau_j:=1/\sqrt{2}$ for all $1\leq j\leq n-1$, $\tau_n:=\tau|\tau|^{-1}\sqrt{2}$.  Then $\mathcal{D}_0$ is equivalent to $\mathcal{C}$ as a fusion category by Lemma \ref{explicit}.  For the braiding, choose a basis $g_1,\ldots,g_n$ of $E_n$ and define $q_j:E_1\times E_1\to\mathbb{C}^\times$ by $q_j(g_j)=q(g_j)$.  This ensures $\alpha^2=\prod_{j=1}^n\alpha_j^2$ as in the proof of Lemma \ref{explicit}, hence $\alpha=\pm\prod_{j=1}^n\alpha_j$.  Therefore if $\alpha_1,\ldots,\alpha_n$ are arbitrary and $\alpha=\prod_{j=1}^n\alpha_j$, we are done, otherwise, switch $\alpha_1\mapsto-\alpha_1$.
\end{proof}

\begin{question}
Does there exist a minimal nondegenerate cover $\mathcal{D}$ of $\mathcal{C}(\chi_n^1,\tau)$ for some $n\in\mathbb{Z}_{\geq1}$ and $\tau=\pm1/2^{n/2}$ which is not braided equivalent to an $n$-fold Deligne product of Ising braided fusion categories?
\end{question}

\section{Braided near-group fusion categories}\label{near}

A fusion category is \emph{near-group} if there exists exactly one isomorphism class of non-invertible objects.  This definition is more general than Tambara-Yamagami since such a fusion category may have a trivial universal grading.  Near-group fusion categories which possess a braiding were classified by J.\ Thornton \cite[Theorem III.4.6]{MR3078486}.  Surprisingly, there are only 7 braided near-group fusion categories up to braided equivalence which are not symmetrically braided, or Tambara-Yamagami (Section \ref{TY}).  These are four of rank 2, two of rank 3, and one of rank 4. 

\begin{note}\label{seitz}
G.\ Seitz \cite{MR222160} classified finite groups having a unique isomorphism class of irreducible representations with dimension greater than 1, giving a complete classification of symmetrically braided near-group fusion categories up to braided equivalence.
\end{note}

\begin{example}\label{buv}
The 7 nonsymmetrically braided near-group group fusion categories which are not Tambara-Yamagami are easily found in nature.  The four examples of rank 2 are described in detail in \cite{ostrik} and all have nondegenerate braidings.  Therefore they are their own unique minimal nondegenerate cover and extension.  They can be constructed from the category $\mathcal{C}(\mathfrak{sl}_2,3)_\mathrm{ad}$ in the notation of \cite{MR4079742}.
\par The two examples of rank 3 have the fusion rules of $\mathrm{Rep}(S_3)$ (but nonsymmetric braidings) where $S_3$ is the symmetric group on $3$ elements, and can be found as braided fusion subcategories of the untwisted double $\mathcal{Z}(\mathrm{Vec}_{S_3})$ $\otimes$-generated, respectively, by the simple objects $(g,\chi)$ where $g$ is any element of order 3 and $\chi$ is one of two nontrivial characters of degree 1 of the cyclic group $C_3$.  Their symmetric center is Tannakian of rank 2 and they each possess 2 inequivalent nondegenerate covers with Frobenius-Perron dimension $12$, which can be realized from $\mathcal{C}(\mathfrak{sl}_2,4)$ in the notation of \cite{MR4079742}.
\par The unique example of rank 4 has the fusion rules of $\mathrm{Rep}(A_4)$ and can be found as a braided fusion subcategory of the untwisted double $\mathcal{Z}(\mathrm{Vec}_{A_4})$ generated by the simple object $(g,\chi)$ where $g$ is any element of order 2 and $\chi$ is any character of degree 1 of $C_2^2$ with $\chi(g)=-1$.  Its symmetric center is Tannakian of rank 3 and it possesses 3 inequivalent nondegenerate covers with Frobenius-Perron dimension $36$, which can be constructed from $\mathcal{C}(\mathfrak{sl}_3,3)$ in the notation of \cite{MR4079742}.
\par The classification of fusion categories with the fusion rules of the rank 3 and 4 examples above dates back to \cite[Section 3]{MR1997336}.
\end{example}

The minimal nondegenerate covers for all the near-group fusion categories discussed so far are weakly integral, hence they have a unique positive spherical structure and can be considered as modular tensor categories (Section \ref{modular}).

\begin{theorem}\label{theorem}
Let $\mathcal{C}$ be a braided near-group fusion category.  Then $\mathcal{C}$ possesses a minimal modular extension if and only if $\mathcal{C}$ is braided equivalent to
\begin{itemize}
\item[\textnormal{(a)}] $\mathcal{Z}(\mathrm{Vec}_{Q_8}^\gamma)_\mathrm{ad}$ or $\mathcal{Z}(\mathrm{Vec}_{Q_8}^\gamma)_\mathrm{ad}^\mathrm{rev}$ for some generator $\gamma\in H^3(Q_8,\mathbb{C}^\times)$,
\item[\textnormal{(b)}] $\mathcal{C}(\chi_n^1,\tau)$ for some $n\in\mathbb{Z}_{\geq1}$ and $\tau=\pm1/2^{n/2}$ with arbitrary braiding data (Section \ref{TY}),
\item[\textnormal{(c)}] a symmetrically braided near-group fusion category (Note \ref{seitz}), or
\item[\textnormal{(d)}] one of the 7 nonsymmetrically braided near-group fusion categories in Example \ref{buv}.
\end{itemize}
\end{theorem}

\begin{proof}
This is a culmination of the results of Section \ref{interesting}, Proposition \ref{bigprop}, \cite[Theorem III.4.6]{MR3078486}, and Example \ref{buv}.
\end{proof}


\section{Extraspecial $p$-groups and minimal nondegenerate covers}\label{futuro}

Here we extend the principal results of Section \ref{jen} to braided fusion categories whose fusion rules coincide with the character rings of extraspecial $p$-groups when $p$ is an odd prime (see Section \ref{extraspecial}).  Fusion and braided fusion categories of these Grothendieck equivalence classes have not been classified as they have for those when $p=2$, so a complete generalization of Section \ref{jen} is left for future research.

\begin{lemma}\label{wan}
Let $\mathcal{C}$ be a braided fusion category Grothendieck equivalent to $\mathrm{Rep}(p^{1+2n}_\pm)$ for an odd prime $p$.  Then $C_\mathcal{C}(\mathcal{C})$ is Tannakian, and either $C_\mathcal{C}(\mathcal{C})=\mathcal{C}$ or $C_\mathcal{C}(\mathcal{C})=\mathcal{C}_\mathrm{pt}$.\end{lemma}

\begin{proof}
We know $C_\mathcal{C}(\mathcal{C})$ is Tannakian by \cite[Corollary 9.9.32(i)]{tcat} since it is symmetrically braided and $p$ is odd.  Any $g\in\mathcal{O}(\mathcal{C}_\mathrm{pt})$ generates a pointed braided fusion subcategory $\mathcal{D}$ of dimension $p$.  Since $p$ is prime, $C_\mathcal{D}(\mathcal{D})=\mathcal{D}$ or $C_\mathcal{D}(\mathcal{D})$ is trivial.  In the latter case $\mathcal{D}$ is nondegenerately braided and $\mathcal{C}$ factors as a nontrivial Deligne product \cite[Theorem 4.2]{mug1}.  But the fusion rules cannot factor since each of the noninvertible objects $\otimes$-generates all of $\mathcal{C}$, thus $\mathcal{C}_\mathrm{pt}$ is symmetrically braided and Tannakian.  Lastly note that when equipped with its unique positive spherical structure, the balancing equation implies that $g\in\mathcal{O}(\mathcal{C}_\mathrm{pt})$ centralize all noninvertible simple objects as they are fixed points of the $\otimes$-action.  Therefore $\mathcal{C}_\mathrm{pt}\subset C_\mathcal{C}(\mathcal{C})$, which proves our claim since $C_\mathcal{C}(\mathcal{C})\subset\mathcal{C}$ is a fusion subcategory.
\end{proof}

\begin{proposition}\label{3andon}
Let $\mathcal{C}$ be a nonsymmetrically braided fusion category with the fusion rules of the character ring of $p^{1+2n}_\pm$ for an odd prime $p$.  If there exists a nondegenerate cover $\mathcal{D}$ of $\mathcal{C}$ with $\mathrm{FPdim}(\mathcal{D})=p^{4n+1}$, then there exists an extraspecial $p$-group $G$ and $3$-cocycle $\omega\in H^3(G,\mathbb{C}^\times)$ such that $\mathcal{C}\simeq\mathcal{Z}(\mathrm{Vec}_G^\omega)_\mathrm{ad}$ is a braided equivalence.
\end{proposition}

\begin{proof}
Lemma \ref{wan} implies that $\mathcal{C}_\mathrm{pt}$ is Tannakian since $p$ is odd; let $R$ be the regular algebra.  Then $\mathrm{FPdim}(\mathcal{D}_R^0)=p^{4n+1}/(p^{2n})^2=p$ \cite[Corollary 3.32]{DMNO}.  Let $\mathcal{P}:=(\mathcal{D}_R^0)^\mathrm{rev}$ so that $\xi(\mathcal{D})=\xi(\mathcal{P})^{-1}$.  This implies \cite[Theorem 1.3]{drinfeld2007grouptheoretical} there exists a finite $p$-group $G$, $\omega\in H^3(G,\mathbb{C}^\times)$, and a braided equivalence $F:\mathcal{D}\boxtimes\mathcal{P}\to\mathcal{Z}(\mathrm{Vec}_G^\omega)$.  As $F$ is monoidal, for all $X\in\mathcal{O}(\mathcal{C})$, $F(X)\in\mathcal{O}(\mathcal{Z}(\mathrm{Vec}_G^\omega)_\mathrm{ad})$.  But Equation (\ref{wanpisu}) implies
\begin{equation}
\mathrm{FPdim}(\mathcal{Z}(\mathrm{Vec}_G^\omega)_\mathrm{ad})=\dfrac{\mathrm{FPdim}(\mathcal{Z}(\mathrm{Vec}_G^\omega))}{\mathrm{FPdim}(\mathcal{Z}(\mathrm{Vec}_G^\omega)_\mathrm{pt})}=\dfrac{p^{4n+2}}{p^{2n+1}}=p^{2n+1}=\mathrm{FPdim}(\mathcal{C}),
\end{equation}
thus the restriction $\left.F\right|_\mathcal{C}:\mathcal{C}\to\mathcal{Z}(\mathrm{Vec}_G^\omega)_\mathrm{ad}$ is a braided equivalence by \cite[Proposition 6.3.3]{tcat}.

\par Finally, recall that $\mathcal{D}_R^0=\mathcal{P}^\mathrm{rev}$.  Each nontrivial simple object of $\mathcal{D}_R^0$ corresponds to the summands of the free $R$-module on some noninvertible $x\in\mathcal{O}(\mathcal{C})$ (see Section \ref{doubles}), so we can derive from this braided equivalence a bijection $\psi:\mathcal{O}(\mathcal{C})\setminus\mathcal{O}(\mathcal{C}_\mathrm{pt})\to\mathcal{O}(\mathcal{P})\setminus\{\mathbbm{1}\}$.  Note that the simple objects $g\in \mathcal{O}(\mathcal{C}_\mathrm{pt})$ and $x\boxtimes\psi(x)$ are closed under $\otimes$ and all have trivial twist, i.e.\ they $\otimes$-generate a Tannakian fusion subcategory of $\mathcal{D}\boxtimes\mathcal{P}$.  This subcategory is Lagrangian, with the fusion rules of the character ring of an extraspecial $p$-group.  Moreover by the reconstruction theorem for twisted doubles of finite groups \cite[Theorem 4.5]{drinfeld2007grouptheoretical}, the finite group $G$ can be chosen to be an extraspecial $p$-group by Lemma \ref{extraspeciallem}.
\end{proof}

\begin{acknowledgements*}
This research was partially funded by the Pacific Institute for the Mathematical Sciences.  We also thank Terry Gannon for his support through the preparation of this manuscript.
\end{acknowledgements*}

\bibliographystyle{plain}
\bibliography{bib}

\begin{thebibliography}{10}

\bibitem{DMNO}
Alexei Davydov, Michael M{\"u}ger, Dmitri Nikshych, and Victor Ostrik.
\newblock The {W}itt group of nondegenerate braided fusion categories.
\newblock {\em Journal f{\"u}r die reine und angewandte Mathematik (Crelles
  Journal)}, (677):135--177, 2013.

\bibitem{drinfeld2007grouptheoretical}
Vladimir Drinfeld, Shlomo Gelaki, Dmitri Nikshych, and Victor Ostrik.
\newblock Group-theoretical properties of nilpotent modular categories.
\newblock arXiv:0704.0195, 2007.

\bibitem{DGNO}
Vladimir Drinfeld, Shlomo Gelaki, Dmitri Nikshych, and Victor Ostrik.
\newblock On braided fusion categories {I}.
\newblock {\em Selecta Mathematica}, 16(1):1--119, 2010.

\bibitem{tcat}
Pavel Etingof, Shlomo Gelaki, Dmitri Nikshych, and Victor Ostrik.
\newblock {\em Tensor Categories}.
\newblock Mathematical Surveys and Monographs. American Mathematical Society,
  2015.

\bibitem{ENO}
Pavel Etingof, Dmitri Nikshych, and Victor Ostrik.
\newblock On fusion categories.
\newblock {\em Annals of Mathematics}, 162(2):581--642, 2005.

\bibitem{MR2677836}
Pavel Etingof, Dmitri Nikshych, and Victor Ostrik.
\newblock Fusion categories and homotopy theory.
\newblock {\em Quantum Topology}, 1(3):209--273, 2010.
\newblock With an appendix by Ehud Meir.

\bibitem{galindo2020trivializing}
César Galindo.
\newblock Trivializing group actions on braided crossed tensor categories and
  graded braided tensor categories.
\newblock arXiv:2010.00847, 2020.

\bibitem{galindo2017categorical}
César Galindo and César~F. Venegas-Ramírez.
\newblock Categorical fermionic actions and minimal modular extensions.
\newblock arXiv:1712.07097, 2017.

\bibitem{gannon1996}
T.~Gannon, P.~Ruelle, and M.~A. Walton.
\newblock Automorphism modular invariants of current algebras.
\newblock {\em Comm. Math. Phys.}, 179(1):121--156, 1996.

\bibitem{2019arXiv191212260G}
Terry {Gannon} and Andrew {Schopieray}.
\newblock Algebraic number fields generated by {F}robenius-{P}erron dimensions
  in fusion rings.
\newblock arXiv:1912.12260, 2019.

\bibitem{nilgelaki}
Shlomo Gelaki and Dmitri Nikshych.
\newblock Nilpotent fusion categories.
\newblock {\em Advances in Mathematics}, 217(3):1053--1071, 2008.

\bibitem{MR2333187}
Christopher Goff, Geoffrey Mason, and Siu-Hung Ng.
\newblock On the gauge equivalence of twisted quantum doubles of elementary
  abelian and extra-special 2-groups.
\newblock {\em J. Algebra}, 312(2):849--875, 2007.

\bibitem{MR1645304}
Bertram Huppert.
\newblock {\em Character theory of finite groups}, volume~25 of {\em De Gruyter
  Expositions in Mathematics}.
\newblock Walter de Gruyter \& Co., Berlin, 1998.

\bibitem{MR3613518}
Tian Lan, Liang Kong, and Xiao-Gang Wen.
\newblock Modular extensions of unitary braided fusion categories and {$2+1{\rm
  D}$} topological/{SPT} orders with symmetries.
\newblock {\em Comm. Math. Phys.}, 351(2):709--739, 2017.

\bibitem{MR3820602}
\'{A}lvaro Mu\~{n}oz and Bernardo Uribe.
\newblock Classification of pointed fusion categories of dimension 8 up to weak
  {M}orita equivalence.
\newblock {\em Comm. Algebra}, 46(9):3873--3888, 2018.

\bibitem{mug1}
Michael M{\"u}ger.
\newblock On the structure of modular categories.
\newblock {\em Proceedings of the London Mathematical Society}, 87(2):291--308,
  2003.

\bibitem{ostrik}
Victor Ostrik.
\newblock Fusion categories of rank 2.
\newblock {\em Mathematical Research Letters}, 10:177--183, 2003.

\bibitem{MR4079742}
Andrew Schopieray.
\newblock Lie theory for fusion categories: {A} research primer.
\newblock In {\em Topological phases of matter and quantum computation}, volume
  747 of {\em Contemp. Math.}, pages 1--26. Amer. Math. Soc., Providence, RI,
  2020.

\bibitem{MR222160}
Gary Seitz.
\newblock Finite groups having only one irreducible representation of degree
  greater than one.
\newblock {\em Proc. Amer. Math. Soc.}, 19:459--461, 1968.

\bibitem{MR1997336}
Jacob Siehler.
\newblock Near-group categories.
\newblock {\em Algebr. Geom. Topol.}, 3:719--775, 2003.

\bibitem{siehler}
Jacob~A. {Siehler}.
\newblock Braided near-group categories.
\newblock arXiv:math/0011037, 2000.

\bibitem{MR1776075}
D.~Tambara.
\newblock Representations of tensor categories with fusion rules of
  self-duality for abelian groups.
\newblock {\em Israel J. Math.}, 118:29--60, 2000.

\bibitem{MR1659954}
Daisuke Tambara and Shigeru Yamagami.
\newblock Tensor categories with fusion rules of self-duality for finite
  abelian groups.
\newblock {\em J. Algebra}, 209(2):692--707, 1998.

\bibitem{MR3078486}
Josiah~E. Thornton.
\newblock {\em Generalized near-group categories}.
\newblock ProQuest LLC, Ann Arbor, MI, 2012.
\newblock Thesis (Ph.D.)--University of Oregon.

\bibitem{MR156890}
C.~T.~C. Wall.
\newblock Quadratic forms on finite groups, and related topics.
\newblock {\em Topology}, 2:281--298, 1963.

\end{thebibliography}

\end{document}